\documentclass[reqno,12pt]{amsart}
\usepackage{bbm}
\usepackage[all,pdf]{xy}
\usepackage{epsfig}
\usepackage{amsmath}
\usepackage{amssymb}
\usepackage{amscd}
\usepackage{graphicx}
\usepackage{color}

\usepackage{mathptmx}

\allowdisplaybreaks[4]

\makeatletter
\@namedef{subjclassname@2020}{%
	\textup{2020} Mathematics Subject Classification}
\makeatother
\usepackage[colorlinks=true]{hyperref}
\usepackage{amsfonts}

\newtheorem{thm}{Theorem}[section]
\newtheorem{lemma}[thm]{Lemma}

\newtheorem{ques}[thm]{Question}

\newtheorem{defn}[thm]{Definition}
\theoremstyle{remark}
\newtheorem{rem}[thm]{Remark}

\newcommand{\norm}[1]{\left\Vert #1\right\Vert}

\def \N {\mathbb N}
\def \C {\mathbb C}
\def \Z {\mathbb Z}
\def \R {\mathbb R}
\def \Q {\mathbb Q}
\def \E {\mathbb E}

\def\B {\mathcal B}

\def \X {\mathcal{X}}
\def \Y {\mathcal{Y}}

\def \a {\alpha }

\def \ep {\epsilon}
\def \d {\delta}

\numberwithin{equation}{section}

\begin{document}

\title[]{pointwise convergence of some continuous-time polynomial ergodic averages}
	
	\author[]{Wen Huang, Song Shao, and Rongzhong Xiao}
	
	\address[Wen Huang]{School of Mathematical Sciences, University of Science and Technology of China, Hefei, Anhui, 230026, PR China}
	\email{wenh@mail.ustc.edu.cn}
	
	\address[Song Shao]{School of Mathematical Sciences, University of Science and Technology of China, Hefei, Anhui, 230026, PR China}
	\email{songshao@ustc.edu.cn}
	
	\address[Rongzhong Xiao]{School of Mathematical Sciences, University of Science and Technology of China, Hefei, Anhui, 230026, PR China}
	\email{xiaorz@mail.ustc.edu.cn}


\begin{abstract}
In this paper, we study the pointwise convergence of centain continuous-time polynomial ergodic averages. Our approach is based on the topological models of measurable flows. One of the main results of this paper is as follows:
Let $a\in \R$, $Q\in \R[t]$ with $\deg Q\ge 2$. Let $(X,\X,\mu, (T^{t})_{t\in \R})$ and $(X,\X,\mu, (S^{t})_{t\in \R})$ be two measurable flows. Then for any $f_1, f_2, g\in L^{\infty}(\mu)$, the limit
\begin{equation*}	\lim\limits_{M\to\infty}\frac{1}{M}\int_{0}^{M}f_1(T^{t}x)f_2(T^{at}x)g(S^{Q(t)}x)dt
\end{equation*}
exists for $\mu$-a.e. $x\in X$. In particular, we are able to build a pointwise ergodic theorem involving geodesic flow and horocycle flow.

\end{abstract}

\maketitle

\section{Introduction}

In this paper, a {\bf measure preserving system} refers to a quadruple $(X,\X, \mu, T)$, where $(X,\X,\mu )$ is a Lebesgue probability space and $T: X \rightarrow X$ is an invertible measure preserving transformation. Let $D\in \N$. A tuple $(X,\X,\mu, (T^{{\bf t}})_{{\bf t}\in \R^{D}})$ is called a \textbf{measurable flow} if $(X,\X,\mu )$ is a Lebesgue probability space, $(T^{\bf t})_{{\bf t}\in \R^{D}}$ is a $D$-parameter group of invertible measure preserving transformations acting on $X$ and the mapping $\R^D\times X\rightarrow X, \ ({\bf t}, x)\mapsto T^{{\bf t}}x$ is measurable. Throughout the paper, for any non-empty finite set $A$ and any function $f$ defined on $A$, we use the notation $$\E_{x\in A}f(x)=\frac{1}{|A|}\sum_{x\in A} f(x)$$ to denote the average of $f$ over $A$, where $|A|$ denotes the cardinality of $A$. For each $N\in \N$, we define $I_{N}=\{0,\ldots, N-1\}$.

\subsection{Furstenberg-Bergelson-Leibman conjecture}

Following Furstenberg's work  \cite{F77} on the dynamical proof of Szemer\'{e}di's theorem, problems concerning the norm convergence and pointwise convergence of
polynomial ergodic averages have garnered significant attention.
A fundamental question in this area is as follows: Given $d,l\in \N$,
let $T_1,\ldots, T_d$ be invertible measure preserving transformations acting on Lebesgue probability space $(X,\X,\mu)$. 
Is it true that for any $f_1,  \ldots,f_l \in L^\infty(\mu)$ and any
$P_{i,j}\in\Z[n],1\le i\le d, 1\le j\le l$, the averages
\begin{equation}\label{MEA-pol-general}
  \E_{n\in I_{N}}\prod_{j=1}^lf_j(T_1^{P_{1,j}(n)}\cdots T_d^{P_{d,j}(n)}x)
\end{equation}
converge in $L^2(\mu)$ and almost everywhere as $N\to\infty$?


It was shown by Bergelson and Leibman in \cite{BVLA} that the polynomial ergodic averages \eqref{MEA-pol-general} may fail to converge in $L^2(\mu)$ when the group generated by $T_1,\ldots, T_d$ is merely assumed to be solvable. In \cite[Page 468]{BVLA}, they conjectured that when $T_1,\ldots, T_d$ generate a nilpotent group,  the polynomial ergodic averages \eqref{MEA-pol-general} converge both in $L^2$-norm and almost everywhere. This conjecture is often referred to as the \textbf{ Furstenberg-Bergelson-Leibman conjecture} in the literature \cite{BMSW,IAMMS, KLMP, BMT}.
Walsh \cite{W} resolved this conjecture for the $L^2$-norm, proving the $L^2$-norm convergence of the averages \eqref{MEA-pol-general} when $T_1,\ldots, T_d$ generate a nilpotent group. Before Walsh's result, there was an
extensive body of research aimed at establishing $L^2$-norm convergence,
including groundbreaking works by Host-Kra \cite{HK05}, Ziegler \cite{Z07}, Tao\cite{T08}, and others. For more details and references, see the survey articles \cite{BV06, FN16, Kra06}.

The problem of pointwise convergence for the expression in \eqref{MEA-pol-general} is significantly more challenging.
The first major breakthrough in this area
was made by Bourgain, who showed that for any $Q\in \Z[n]$ and any $f\in L^\infty(\mu)$, the averages $$\E_{n\in I_{N}}f(T_{1}^{Q(n)}x)$$ converge almost everywhere as $N\to\infty$ \cite{B88}; and for any $a_1, a_2\in \Z$ and any $f_1, f_2 \in  L^\infty(\mu)$,
the averages $$\E_{n\in I_{N}} f_1(T_{1}^{a_1n}x)f_2(T_{1}^{a_2n}x)$$  converge almost everywhere as $N\to\infty$ \cite{B4}. For more results on bilinear ergodic averages, one can see \cite{ELH,X,krause2025}.

In 2022, Krause, Mirek and Tao proved that for any $Q\in \Z[n]$ with $\deg Q\ge 2$ and any $f_1, f_2 \in  L^\infty(\mu)$, the averages $$\E_{n\in I_{N}}f_1(T_{1}^nx)f_2(T_{1}^{Q(n)}x)$$ converge almost everywhere as $N\to\infty$ \cite{BMT}. In 2023, Ionescu, Magyar,  Mirek and Szarek \cite{IAMMS} proved that if $T_1, \ldots, T_d$ generate a nilpotent group of step two, then for any $Q_1,\ldots,Q_d\in \Z[n]$ and any $f \in  L^\infty(\mu)$, the averages $$\E_{n\in I_{N}}f(T_1^{Q_{1}(n)}\cdots T_d^{Q_{d}(n)}x)$$ converge almost everywhere as $N\to\infty$. In 2024, Kosz, Mirek, Peluse and Wright \cite{KMPW24} extended the result of Krause, Mirek and Tao by proving that if $T_1, \ldots, T_d$ are commuting and $Q_1,\ldots,Q_d\in \Z[n]$ with distinct degrees, then for any $f_1,\ldots f_d\in L^{\infty}(\mu)$, the averages $$\E_{n\in I_{N}}f_1(T_{1}^{Q_{1}(n)}x)\cdots f_d(T_{d}^{Q_{d}(n)}x)$$ converge almost everywhere as $N\to\infty$. 

When $l>2$, the pointwise convergence of the expression in \eqref{MEA-pol-general} can only be considered for certain special classes of measure preserving systems at this stage. For instance, when $(X,\X,\mu,T_1)$ is a $K$-system, Derrien and Lesigne proved that for any $Q_1,\ldots, Q_d\in \Z[n]$ and any $f_1,\ldots, f_d\in L^\infty(\mu)$, the averages $$\E_{n\in I_{N}} f_1(T_{1}^{Q_{1}(n)}x)\cdots f_d(T_{1}^{Q_{d}(n)}x)$$ converge almost everywhere as $N\to\infty$ \cite{DL96}. In 2024, the third author \cite{Xiao2024} extended the result of Derrien and Lesigne to $\Z^d$-actions ($d\ge 1$). When $(X,\X,\mu,T_1)$ is an ergodic distal system, Huang, Shao and Ye proved that for any $f_1,\ldots, f_d\in L^\infty(\mu)$, the averages $$\E_{n\in I_{N}} f_1(T_{1}^{n}x)\cdots f_d(T_{1}^{dn}x)$$ converge almost everywhere as $N\to\infty$ \cite{HSY19}. When $(X,\X,\mu,T_1,\ldots,T_d)$ is an ergodic distal system and $T_1,\ldots, T_d$ are commuting, Donoso and Sun extended the result of Huang, Shao and Ye by proving that for any $f_1,\ldots, f_d\in L^\infty(\mu)$, the averages $$\E_{n\in I_{N}} f_1(T_{1}^{n}x)\cdots f_d(T_{d}^{n}x)$$ converge almost everywhere as $N\to\infty$ \cite{DS}.

\subsection{Continuous-time analogues of Furstenberg-Bergelson-Leibman conjecture}

In \cite{BLM}, Bergelson, Leibman and Moreira introduced the method, which allows ones to derive continuous-time versions of various discrete-time ergodic theorems. For example, combining the mothod with Bourgain's double recurrence theorem and polynomial ergodic theorem, it was shown that for a measurable flow $(X,\X,\mu, (T^{t})_{t\in \R})$, any $f, g\in L^{\infty}(\mu)$, any $a_1,a_2\in \Z$ and any $P\in \R[t]$, the limits $$\lim\limits_{M\to\infty}\frac{1}{M}\int_{0}^{M}f(T^{a_1t}x)g(T^{a_2t}x)dt\ \text{and}\ \lim\limits_{M\to\infty}\frac{1}{M}\int_{0}^{M}f(T^{P(t)}x)dt$$ exist for $\mu$-a.e. $x\in X$ \cite[Theorem 8.30, 8.31]{BLM}.



In 2022, Bourgain, Mirek, Stein and Wright proved the pointwise convergence for certain multi-parameter polynomial ergodic averages \cite{BMSW}. They showed that for a measure preserving system $(X,\X,\mu,T)$, any $k\in\N$, any $P\in \Z[n_1,\ldots, n_k]$ and any $f\in L^\infty(\mu)$, the limit
\begin{equation*}
  \lim_{\min\{N_1,\ldots, N_k\}\to\infty} \E_{(n_1,\ldots, n_k)\in \prod_{j=1}^k I_{N_j}} f(T^{P(n_1,\ldots, n_k)}x)
\end{equation*}
exists for $\mu$-a.e. $x\in X$. Let $T_1,\ldots, T_d:X\rightarrow X$ be invertible, commuting measure preserving transformations acting on Lebesgue probability space $(X,\X,\mu)$. The {\bf Furstenberg-Bellow problem} conjectures that for any $k\in\N$, any $P_1,\ldots, P_d\in \Z[n_1,\ldots, n_k]$ and any $f\in L^\infty(\mu)$, the limit
\begin{equation*}
  \lim_{\min\{N_1,\ldots, N_k\}\to\infty} \E_{{\bf n}\in \prod_{j=1}^k I_{N_j}} f(T_1^{P_1({\bf n})}\cdots T_d^{P_d({\bf n})}x)
\end{equation*}
exists for $\mu$-a.e. $x\in X$, where ${\bf n}=(n_1,\ldots, n_k)$
\cite[Conjecture 1.22]{BMSW}. This conjecture remains open.

Recently, Kosz, Langowski, Mirek and Plewa provided an affirmative answer to its continuous-time analogues \cite{KLMP}. The specific statement is as follows: Given $D,k\in\N$, let $(X,\X,\mu, (T^{\bf t})_{{\bf t}\in \R^{D}})$ be a measurable flow. Then for any $P_1,\ldots,P_D\in  \R[t_1,\ldots,t_k]$ and any $f\in L^{\infty}(\mu)$, the limit
\begin{equation*}
  \lim_{\min\{M_1,\ldots, M_k\}\to\infty} \frac{1}{M_{1}\cdots M_{k}}\int_{\prod_{j=1}^{k}[0,M_j]}f(T^{P_{1}({\bf t})\vec{e}_1}\cdots T^{P_{D}({\bf t})\vec{e}_D}x)d{\bf t}
\end{equation*}
exists for $\mu$-a.e. $x\in X$, where for each $j\in\{1,\ldots,D\}$, $\vec{e}_j$ denotes the $j$-th basis vector of the standard basis in $\R^{D}$. Similar to the \textbf{Furstenberg-Bergelson-Leibman conjecture}, for measurable flows, we have the corresponding question:

\begin{ques}
Given $D\in\N$, let $(X,\X,\mu, (T^{\bf t})_{{\bf t}\in \R^{D}})$ be a measurable flow. Let $k,l\in \N$. Then for any $P_{i,j}\in  \R[t_1,\ldots,t_k], 1\le i\le D, 1\le j\le l$ and any $f_1,\ldots, f_l\in L^{\infty}(\mu)$, does the limit
\begin{equation}\label{aa3}
  \lim_{\min\{M_1,\ldots, M_k\}\to\infty} \frac{1}{M_{1}\cdots M_{k}}\int_{ \prod_{i=1}^{k}[0,M_i]}\prod_{j=1}^lf_j(T^{P_{1,j}({\bf t})\vec{e}_1}\cdots T^{P_{D,j}({\bf t})\vec{e}_D}x)d{\bf t}
\end{equation}
exist for $\mu$-a.e. $x\in X$?
\end{ques}
In this paper, we will consider some special cases of \eqref{aa3}.
\begin{rem}
When $k=1$, Austin \cite{A} proved that the limit in \eqref{aa3} exists in $L^{2}(\mu)$. When $D=1$, Potts \cite{P} showed that the limit in \eqref{aa3} exists in $L^{2}(\mu)$. For the general case, by combining \cite[Theorem 6.14]{BLM} and \cite[Theorem 5.1]{W}, one can show that the limit in \eqref{aa3} exists in $L^{2}(\mu)$.
\end{rem}
\subsection{Ergodic theorems without commutativity} In 2023, Frantzikinakis and Host presented the following mean ergodic theorem, which deos not require any commutativity.

\medskip

\begin{thm}[{\cite[Theorem 1.1]{HN}}]\label{thm-FH}
Let $T,S$ be invertible measure preserving transformations acting on Lebesgue probability
space $(X,\X,\mu)$ such that $(X,\X,\mu,T)$ has zero entropy. Let $P\in \Z[n]$ with $\deg P \ge  2$. Then for any $f,g\in L^\infty(\mu)$, the limit
$$
  \lim_{N\to\infty}\E_{n\in I_{N}} f(T^nx)g(S^{P(n)}x)
$$
exists in $L^2(\mu)$.
\end{thm}
The zero entropy assumption on
$T$ is necessary in the above result. If both $T$ and $S$ have positive entropy, the above result may fail (See \cite[Proposition 1.4]{HN} for details.).

Up to the present, it remains unknown whether the limit in Theorem \ref{thm-FH} exists almost everywhere.
But in the continuous-time setting, Frantzikinakis showed the following result:
Let $(X,\X,\mu, (T^{t})_{t\in \R})$ and $(X,\X,\mu, (S^{t})_{t\in \R})$ be two measurable flows, with no commuting restrictions imposed on the $\R$-actions $(T^{t})_{t\in \R}$ and $(S^{t})_{t\in \R}$. Then  for any $f, g\in L^{\infty}(\mu)$, the limit $$\lim\limits_{M\to\infty}\frac{1}{M}\int_{0}^{M}f(T^{t}x)g(S^{t^{2}}x)dt$$ exists for $\mu$-a.e. $x\in X$ (a special case of \cite[Theorem 1.11]{FN}).
When $(T^{t})_{t\in \R}$ and $(S^{t})_{t\in \R}$ are commuting $\R$-actions, this result was  previously established by Christ, Durcik, Kova{\v{c}}, and Roos in \cite{CDKR}.

In \cite{HSY23,X}, Theorem \ref{thm-FH} was extended to the following form:
Let $T,S$ be invertible measure preserving transformations acting on Lebesgue probability
space $(X,\X,\mu)$ such that $(X,\X,\mu,T)$ has zero entropy. Then for any $a_1,a_2\in\Z, P\in \Z[n]$ with $\deg P\ge 2$ and any $f,g\in L^\infty(\mu)$, the limit
	\begin{equation}\label{aa1}
	 \lim_{N\to\infty}\E_{n\in I_{N}}f_{1}(T^{a_1n}x)f_{2}(T^{a_2n}x)g(S^{P(n)}x)
	\end{equation}
exists in $L^2(\mu)$. We conjecture that the limit in \eqref{aa1} exists almost everywhere.
In this paper, we will consider some continuous-time analogues of the limit in \eqref{aa1}.

\subsection{Main results}

Before presenting the main results of the paper, we introduce two notations. Let $D\in\N$. For a measurable flow $(X,\X,\mu, (T^{{\bf t}})_{{\bf t}\in \R^{D}})$, we define $\mathcal{I}((T^{\bf t})_{{\bf t}\in\R^{D}})$ as the sub-$\sigma$-algebra of $\X$ generated by $\{A\in\X:\text{for}\ \text{each}\  {\bf t}\in \R^{D},T^{{\bf t}}A=A\ (\text{mod}\ \mu)\}$. For a Lebesgue probability space $(Y,\Y,\nu)$, a sub-$\sigma$-algebra $\mathcal{F}$ of $\Y$ and any $f\in L^{1}(Y,\Y,\nu)$, $\E_{\nu}(f|\mathcal{F})$ denotes the conditional expectation of $f$ with respect to $\mathcal{F}$ (For the definition, see Subsection \ref{CE}.) .

The following result shows some continuous-time analogues of the limit in \eqref{aa1} exist almost everywhere.

\medskip
\noindent{\bf Theorem A.}\ {\em
Let $(X,\X,\mu, (T^{t})_{t\in \R})$ and $(X,\X,\mu, (S^{t})_{t\in \R})$ be two measurable flows, $a\in \R$, and $Q\in \R[t]$ with $\deg Q\ge 2$. Then for any $f_1, f_2, g\in L^{\infty}( \mu)$, the limit
\begin{equation}\label{aa2}
	\lim\limits_{M\to\infty}\frac{1}{M}\int_{0}^{M}f_1(T^{t}x)f_2(T^{at}x)g(S^{Q(t)}x)dt
\end{equation}
exists for $\mu$-a.e. $x\in X$.

} 	
\medskip


\medskip

In fact, we can show more.

\medskip
\noindent{\bf Theorem B.}\ {\em
			Let $(X,\X,\mu, (T^{t})_{t\in \R})$ and $(X,\X,\mu, (S^{t})_{t\in \R})$ be two measurable flows. Let $ a\in \mathbb{R}$, $Q\in \R[t]$ with $\deg Q\ge 2$, and $0<\alpha\le \beta$. Then for any $f_1,f_2, g\in L^{\infty}(\mu)$, the limit
\begin{equation}\label{TA1}					\lim\limits_{M\to\infty}\frac{1}{M}\int_{0}^{M}f_{1}(T^{t^\alpha}x)f_{2}(T^{at^\alpha}x)g(S^{Q(t^{\beta})}x)dt
\end{equation}
exists for $\mu$-a.e. $x\in X$. We denote the limit function by $L(f_1,f_2,g)$. Then for $\mu$-a.e. $x\in X$,
\begin{equation}\label{zz1}
L(f_1,f_2,g)=\E_{\mu}(g|\mathcal{I}((S^t)_{t\in \R}))(x)\lim\limits_{M\to\infty}\frac{1}{M}\int_{0}^{M}f_{1}(T^{t}x)f_{2}(T^{at}x)dt.
\end{equation}
}


\begin{rem}\label{rem1}
\begin{enumerate}
	\item The limit $\displaystyle \lim\limits_{M\to\infty}\frac{1}{M}\int_{0}^{M}f_{1}(T^{t}x)f_{2}(T^{at}x)dt$ appearing on the right-hand side of \eqref{zz1} is guaranteed to exist almost everywhere by Theorem \ref{thm6}. For the case where $a\in \Q$, further details regarding the properties of the associated limit function can be found in \cite{BLM,P}.

	\item When $\deg Q=1$, Theorem B may fail. In \cite[Subsection 4.1]{BVLA}, Bergelson and Leibman constructed a Lebesgue probability space $(Y,\mathcal{D},\nu)$, two invertible measure preserving transformations $S,R:Y\rightarrow Y$ and a positive measure subset $A$ such that the limit $$\lim_{N\to\infty}\E_{n\in I_N}\nu(S^{-n}A\cap R^{-n}A)$$ does not exists. Let $Z=Y\times [0,1)$. For any $t\in \R$, we define $\tilde{S}^{t}:Z\rightarrow Z,(y,s)\mapsto (S^{\lfloor t+s\rfloor}x,(t+s)\ \text{mod}\ 1)$. Similarly, we can define $\tilde{R}^{t}$ for any $t\in \R$. Clearly, $(Z,\mathcal{D}\otimes \B([0,1)),\nu\times m, (\tilde{S}^{t})_{t\in \R})$ and $(Z,\mathcal{D}\otimes \B([0,1)),\nu\times m, (\tilde{R}^{t})_{t\in \R})$ are two measurable flows, where $m$ denotes the Lebesgue measure on $[0,1)$ and $\B([0,1))$ is the Borel $\sigma$-algebra of $[0,1)$. Let $B=A\times [0,1)$. Then the limit
	$$\lim\limits_{M\to\infty}\frac{1}{M}\int_{0}^{M}1_{B}(\tilde{S}^{t}x)1_{B}(\tilde{R}^{t}x)dt$$ does not exist in $L^{2}(\nu\times m)$.
\end{enumerate}		
\end{rem}

With the same method, we can establish the following results:

 \medskip
\noindent{\bf Theorem C.}\ {\em
Let $(X,\X,\mu, (T^{t})_{t\in \R})$ and $(X,\X,\mu, (S^{t})_{t\in \R})$ be two measurable flows. Let $c\in \R$. For any $k\in \N$ , any $P\in \Z[t_1,\ldots,t_k]$ with $\deg P=1$, and any $f, g\in L^{\infty}(\mu)$, the limit
\begin{equation}\label{TA2}
\lim\limits_{\min\{M_1,\ldots,M_k\}\to\infty}\frac{1}{M_{1}\cdots M_{k}}\int_{\prod_{j=1}^{k}[0,M_j]}f(T^{{|{\bf t}|}}x)g(S^{{|{\bf t}|}^{2}+cP({\bf t})}x)d{\bf t}
\end{equation}
exists for $\mu$-a.e. $x\in X$, where ${|{\bf t}|}=|t_1|+\cdots +|t_k|$. We denote the limit function by $L(f,g)$. Then for $\mu$-a.e. $x\in X$, $$L(f,g)=\E_{\mu}(f|\mathcal{I}((T^t)_{t\in \R}))(x)\E_{\mu}(g|\mathcal{I}((S^t)_{t\in \R}))(x).$$
		
}

\medskip

\begin{rem}\label{rem-1}
When $f$ or $g$ is constant, the fact that the limit in \eqref{TA2} exists almost everywhere can be deduced from \cite[Theorem 1.5.(ii)]{KLMP} directly.
\end{rem}


\medskip

\noindent{\bf Theorem D.} \ {\em
Let $(X,\X,\mu, (S^{\bf t})_{{\bf t}\in \R^{2}})$ be a measurable flow. Let $(X,\X,\mu, (T_{i}^{t})_{t\in \R}),1\le i\le d$ be measurable flows, where $d\in \N$. Let $\vec{e}_1=(1,0)$ and $\vec{e}_2=(0,1)$. Let $0<\alpha_1<\cdots<\alpha_d<\beta$,  $c\in \mathbb{R}$, and $Q\in \R[t]$ with $\deg Q\ge 2$. Then the following hold:
\begin{itemize}
\item[(1)] For any $f_1,\ldots,f_d,g\in L^{\infty}(\mu)$, the limit
\begin{equation}\label{TB1}
 \lim\limits_{M\to\infty}\frac{1}{M}\int_{0}^{M}f_{1}(T_{1}^{t^{\alpha_1}}x)\cdots f_{d}(T_{d}^{t^{\alpha_d}}x)g(S^{Q(t^{\beta})\vec{e}_1}S^{t^{\beta}\vec{e}_2}x)dt
\end{equation} exists for $\mu$-a.e. $x\in X$. We denote the limit function by $L(f_1,\dots,f_d,g)$. Then for $\mu$-a.e. $x\in X$, $$L(f_1,\dots,f_d,g)=\E_{\mu}(g|\mathcal{I}((S^{\bf t})_{{\bf t}\in \R^{2}}))(x)\prod_{j=1}^{d}\E_{\mu}(f_{j}|\mathcal{I}((T_{j}^t)_{t\in \R}))(x).$$

\item[(2)] For any $f,g\in L^{\infty}(\mu)$, the limit
\begin{equation}\label{TB2}
\lim\limits_{M\to\infty}\frac{1}{M}\int_{0}^{M}f(S^{ct^{\beta}\vec{e}_2}x)g(S^{Q(t^{\beta})\vec{e}_1}S^{t^{\beta}\vec{e}_2}x)dt
\end{equation}
exists for $\mu$-a.e. $x\in X$. We denote the limit function by $L(f,g)$. Then for $\mu$-a.e. $x\in X$, $$L(f,g)=\lim\limits_{M\to\infty}\frac{1}{M}\int_{0}^{M}f(S^{ct\vec{e}_2}x)\E_{\mu}(g|\mathcal{I}((S^{t\vec{e}_1})_{{t}\in \R}))(S^{t\vec{e}_2}x)dt.$$
\end{itemize}
}

To obtain a concrete application, let us focus on the geodesic flow and horocycle flow. After combining the method used in the proof of Theorem B, \cite[Main theorem]{GHSY} and \cite[Section 6 of Chapter 4]{G}, we are able to establish the following pointwise ergodic theorem:

\medskip

\noindent{\bf Corollary E.}\ {\em
	Let $X=SL_{2}(\R)/SL_{2}(\Z)$ and $m$ be the Harr measure on $X$. Given $d\in \N$, let $ c_1,\ldots,c_d$ be distinct non-zero rational numbers. Let $Q\in \R[t]$ with $\deg Q\ge 2$. Then for any $f_1,\ldots,f_d,g\in L^{\infty}(m)$ and $m$-a.e. $x\in X$,
	$$\lim\limits_{M\to\infty}\frac{1}{M}\int_{0}^{M}g(a(Q(t))x)\prod_{j=1}^{d}f_{j}(u(c_{j}t)x)dt=\int_{X}gdm\prod_{j=1}^{d}\int_{X}f_jdm,$$ where for each $t\in \R$,
	$$
	a(t)=\begin{pmatrix}
	e^t & 0 \\
	0 & e^{-t}
	\end{pmatrix}, \quad
	u(t)=\begin{pmatrix}
	1& t\\
	0 & 1
	\end{pmatrix}.
	$$
}

\medskip

\begin{rem}
The absence of a general multi-linear version of Theorem \ref{thm-K} necessitates the restriction that $ c_1,\ldots,c_d\in \Q$ in the above corollary. However, it is worth noting that when $d=2$, the corollary remains valid for all distinct $c_1,c_2\in \R$.
\end{rem}

\subsection{A brief overview of the proofs}

We now provide a concise outline of the proof structure. Initially, we simplify the problem by reducing bounded measurable functions to a class of typical functions. These typical functions can be regarded as continuous functions defined on specific compact metric spaces. This reduction is facilitated through the application of topological models of measurable flows and relevant maximal inequalities. Subsequently, we transform the continuous-time ergodic averages under consideration into discrete-time ergodic averages, leveraging the continuity properties of the typical functions. Finally, we establish the pointwise convergence of these discrete-time ergodic averages by established ergodic theorems.

To help the reader's comprehension, we present a detailed sketch of the proof for a special case of Theorem A in Subsection \ref{sample}.


\subsection*{Organization of the paper}

The paper is organized as follows. In Section \ref{Section-pre}, we introduce some necessary notions and review some known results used in the paper. In Section \ref{section-BC}, we prove Theorem B and Theorem C, and in Section \ref{section-D}, we show Theorem D. In Section \ref{section-ques}, we ask a question. Finally, Appendix \ref{Ap1} includes a new ergodic theorem.

\subsection*{Acknowledgement}
The first author and second author are supported by National Key R$\&$D Program of China (2024YFA1013601, 2024YFA1013602, 2024YFA1013600), and National Natural Science Foundation of China (12426201, 12371196, 12201599). The third author is supported by National Natural Science Foundation of China (123B2007, 12371196).

	
\section{Preliminaries}\label{Section-pre}
	
In this section, we give some basic notations used throughout the paper, and we also introduce some results that will be employed in subsequent discussions.

\subsection{Notations}

\begin{itemize}
  \item The set of integers (resp. natural numbers $\{1,2,\ldots\}$, rational numbers, real numbers, complex numbers) is denoted by $\Z$ (resp. $\N$, $\Q$, $\R$, $\C$).
Let $\R_{+}=\{x\in \R:x>0\}$.
  \item Let $k\in\N$. In this paper, we use ${\bf n}$ and ${\bf t}$ to denote the element of $\Z^k$ and $\R^k$, respectively. For any ${\bf t}\in \R^{k}$, let $|{\bf t}|=|t_1|+\cdots + |t_k|$. Let ${\bf 0}$ be the original point of $\R^{k}$.
  \item For any $x\in \R$, $\lfloor x \rfloor=n$, where $n\le x < n+1, n\in\Z$.
\item For every $N\in \N$, let $I_N=\{0,\ldots,N-1\}$.

\item For any non-empty finite set $A$ and any function $f$ on $A$, we write $$\E_{x\in A}f(x)=\frac{1}{|A|}\sum_{x\in A} f(x)$$ for the average of $f$ over $A$, where $|A|$ denotes the cardinality of $A$.

\item For a measure preserving system $(X,\X,\mu,T)$, $\mathcal{I}(T)$ denotes the sub-$\sigma$-algebra of $\X$ generated by $\{A\in\X:T^{-1}A=A\ (\text{mod}\ \mu)\}$.

\item Let $D\in\N$. For a measurable flow $\left(X,\X,\mu, (T^{{\bf t}})_{{\bf t}\in \R^{D}}\right)$, $\mathcal{I}\big((T^{\bf t})_{{\bf t}\in\R^{D}}\big)$ denotes the sub-$\sigma$-algebra of $\X$ generated by $\{A\in\X:\text{for}\ \text{each}\  {\bf t}\in \R^{D},T^{{\bf t}}A=A\ (\text{mod}\ \mu)\}$.
\end{itemize}

\subsection{Conditional expectation and disintegration of a measure}\label{CE}
For Lebesgue probability space $(Y,\Y,\nu)$, any sub-$\sigma$-algebra $\mathcal{F}$ of $\Y$ and any $f\in L^{1}(Y,\Y,\nu)$, the \textbf{conditional expectation} of $f$ with respect to ${\bf \mathcal{F}}$ is denoted by  $\E_{\nu}(f|\mathcal{F}) $, and it is defined in $L^{1}(Y,\mathcal{F},\nu)$ such that for any $A\in \mathcal{F}$, $$\int_{A}fd\nu=\int_{A}\E_{\nu}(f|\mathcal{F})d\nu.$$ Then there exists a unique $\Y$-measurable map $Y\to \mathcal{M}(Y,\Y),y\mapsto \nu_y$, called the {\bf disintegration} of $\nu$ with respect to $\mathcal{F}$, under neglecting $\nu$-null sets such that for any $f\in L^{\infty}(Y,\Y,\nu)$, $$\E_{\nu}(f|\mathcal{F})(y)=\int_{Y}fd\nu_y$$ for $\nu$-a.e. $y\in Y$, where $\mathcal{M}(Y,\Y)$ is the collection of probability measures on $(Y,\Y)$, endowed with standard Borel structure.

\subsection{Topological models of measurable flows}

Let $(X,\X,\mu, (T^{{\bf t}})_{{\bf t}\in \R^{D}})$ be a measurable flow, where $D\in \N$ . A set $X_0\subset X$ is $(T^{\bf t})_{{\bf t}\in \R^{D}}$-invariant if $T^{\bf t}x\in X_0$ for every $x\in X_0, {\bf t}\in\R^{D}$.
	
	Two measurable flows $(X,\X,\mu, (T^{\bf t})_{{\bf t}\in \R^{D}})$ and $(Y,\Y,\nu, (S^{{\bf t}})_{{\bf t}\in \R^{D}})$ are \textbf{isomorphic} if there is an invertible measure preserving map $\phi$ between $(T^{\bf t})_{{\bf t}\in \R^{D}}$-invariant measurable subset $X_0\subset X$ and  $(S^{\bf t})_{{\bf t}\in \R^{D}}$-invariant measurable subset $Y_0\subset Y$, both of measure $1$, such that $$\phi\circ T^{\bf t}(x)=S^{\bf t}\circ \phi(x)\ \text{for}\ \text{each}\  {\bf t}\in\R^{D},x\in X_0.$$
	This is illustrated by the following commutative diagram:
\[
\begin{CD}
X_0 @>{T^{\bf t}}>> X_0\\
@V{\phi}VV      @VV{{\phi}}V\\
Y_0 @>{S^{\bf t} }>> Y_0.
\end{CD}
\]
	
Let $Y$ be a compact metric space. A pair $(Y,(S^{\bf t})_{{\bf t}\in \R^{D}})$ is called a \textbf{continuous flow} if $(S^{\bf t})_{{\bf t}\in \R^{D}}$ is a $D$-parameter group of homeomorphisms acting on $Y$ and the mapping $\R^D\times Y\rightarrow Y,\ ({\bf t},y)\mapsto S^{\bf t}y$ is continuous. Let $\B(Y)$ be the Borel $\sigma$-algebra of $Y$, and let $\nu$ be an $(S^{\bf t})_{{\bf t}\in \R^{D}}$-invariant Borel probability measure on $Y$. Then $(Y,\B(Y),\nu,(S^{\bf t})_{{\bf t}\in\R^{D}})$ is a measurable flow.
	
\begin{defn}
A measurable flow $(Y,\B(Y),\nu,(S^{\bf t})_{{\bf t}\in\R^{D}})$ is called a \textbf{topological model} of the measurable flow $(X,\X,\mu, (T^{\bf t})_{{\bf t} \in \R^{D}})$, if $(Y,(S^{\bf t})_{{\bf t}\in\R^{D}})$ is a continuous flow, $\nu$ is an $(S^{\bf t})_{{\bf t}\in \R^{D}}$-invariant Borel probability measure on $Y$, and the two measurable flows $(X,\X,\mu, (T^{\bf t})_{{\bf t}\in \R^{D}})$ and $(Y,\B(Y),\nu,(S^{\bf t})_{{\bf t}\in\R^{D}})$ are isomorphic.
\end{defn}
	
\begin{thm} [{\cite[Theorem 3.2]{V}}] \label{thm1}
Every measurable flow has a topological model.
\end{thm}

\subsection{Maximal ergodic theorems}
	
Let $(X,\X,\mu, (T^{\bf t})_{{\bf t}\in \R^{D}})$ be a measurable flow, where $D\in \N$. Let $\mathcal{P}=\{P_1,\ldots,P_D\}\subset \R[t_1,\ldots,t_k]$, where $k\in \N$. For any $M_1,\ldots,M_k\in \R_{+}$ and any $f\in L^{2}(\mu)$, let
$$
A(\mathcal{P};M_1,\ldots,M_k;f)(x)=\frac{1}{M_{1}\cdots M_{k}}\int_{{ \prod_{j=1}^{k}[0,M_j]}}f(T^{P_{1}({\bf t})\vec{e}_1}\cdots T^{P_{D}({\bf t})\vec{e}_D}x)d{\bf t},
$$
where for each $j\in\{1,\ldots,D\}$, $\vec{e}_j$ denotes the $j$-th basis vector of the standard basis in $\R^{D}$.
	
\begin{thm}[{\cite[Theorem 1.5.(iii)]{KLMP}}]\label{thm2}
There is a constant $C=C(\deg \mathcal{P})>0$ such that
	 	\begin{equation}\label{eq5}
	 		\norm{\sup_{M_1,\ldots,M_k\in \R^{+}}\big|A(\mathcal{P};M_1,\ldots,M_k;f)(x)\big|}_{2}\le C\norm{f}_{2},
	 	\end{equation}
 where $\deg \mathcal{P}=\max\{\deg P_j:1\le j\le D\}$.
\end{thm}



We will need a special case of the following theorem in the proofs.
\begin{thm}[{\cite[Theorem 1.2.-(iii)]{IAMMS}}]\label{thm7}
			Let $d_1\in \N$ be given. Let $T_1,\ldots,T_{d_1}:X\rightarrow X$ be a family of invertible measure preserving transformations acting on Lebesgue probability space $(X,\X,\mu)$, that generates a nilpotent group of step two. Assume that $P_1,\ldots,P_{d_1}\in \Z[n]$ and let $d_2=\max\{\deg P_j(n):1\le j\le d_1\}$. Then there exists a constant $C=C(d_1,d_2)$ such that for any $f\in L^{2}(\mu)$,
			\begin{equation}\label{eq38}
				\norm{\sup_{N\ge 1}\Big|\E_{n\in I_N}f(T_{1}^{P_1(n)}\cdots T_{d_1}^{P_{d_1}(n)}x)\Big|}_{2}\le C\norm{f}_{2}.
			\end{equation}
		\end{thm}

\subsection{Pointwise ergodic theorems for discrete-time actions}
First, we introduce Bourgain's polynomial ergodic theorem and a double recurrence theorem due to Krause.
\begin{thm}[{\cite[Theorem 2]{B2}}]\label{BPET}
	Let $(X,\X,\mu, T)$ be a measure preserving system. For any $P\in \R[n]$ and any $f\in L^{\infty}(\mu)$, the limit
	\begin{equation}
	\lim\limits_{N\to\infty}\E_{n\in I_N}f(T^{\lfloor P(n)\rfloor}x)
	\end{equation}
exists almost everywhere.
\end{thm}


In 2025, Krause gave an extension of Bourgain's double recurrence theorem \cite[Main Theorem]{B4}.

\begin{thm}[{\cite[Theorem 1.1]{krause2025}}] \label{thm-K}
Let $(X,\X,\mu,T)$ be a measure preserving system. Then for all $\alpha\in \R$, all $\gamma\in \Q$, and all $f_1,f_2\in L^{\infty}(\mu)$,
\begin{equation}\label{eqK}
\lim\limits_{N\to\infty}\E_{n\in I_N}f_{1}(T^{\lfloor \alpha n\rfloor}x)f_{2}(T^{\lfloor \gamma n\rfloor}x)
\end{equation}
exists almost everywhere.
\end{thm}

The following theorem can be viewed as a multi-parameter extension of pointwise ergodic theorem along polynomials with integer coefficients \cite[Theorem 1]{B2}. 

\begin{thm}[{\cite[Theorem 1.5]{BMSW}, \cite{D51, Z51}}]\label{thm3-2}
Let $(X,\X,\mu,T)$ be a measure preserving system. Then for any $k\in \N$, any $P\in \Z[n_1,\ldots,n_k]$ and any $f\in L^{\infty}(\mu)$, the limit
\begin{equation}\label{eq8}
 \lim_{\min\{N_1,\ldots, N_k\}\to\infty} \E_{(n_1,\ldots, n_k)\in \prod_{j=1}^k  I_{N_j}} f(T^{P(n_1,\ldots, n_k)}x)
\end{equation}
exists almost everywhere. If $P(n_1,\ldots,n_k)=n_{1}+\cdots+n_{k}$, then the related limit function is $\E_{\mu}(f|\mathcal{I}(T))$.
\end{thm}

We will need a special case of the following theorem in the proofs.

\begin{thm}[{\cite[Theorem 1.2.(ii)]{IAMMS}}]\label{thm3-3}
Let $(X,\X,\mu)$ be a Lebesgue probability space. Let $T_1,\ldots, T_d:X\rightarrow X$ be a family of invertible measure preserving transformations, that generates a nilpotent group of step two. Then for any $P_1,\ldots, P_d\in \Z[n]$ and any $f\in L^{\infty}(\mu)$, the limit
\begin{equation}\label{eq9} \lim\limits_{N\to\infty}\E_{n\in I_N}f(T_{1}^{P_{1}(n)}\cdots T_{d}^{P_{d}(n)}x)
\end{equation}
exists almost everywhere.
\end{thm}

\subsection{Pointwise ergodic theorems for measurable flows}

	
First,  we apply Krause's double recurrence theorem (Theorem \ref{thm-K}) to derive the following result.

\begin{thm}\label{thm6}
		Let $(X,\X,\mu, (T^{t})_{t\in \R})$ be a measurable flow. Then for any $a\in \R$ and any $f_1,f_2\in L^{\infty}(\mu)$, the limit
		\begin{equation}\label{eq26}
		\lim\limits_{M\to\infty}\frac{1}{M}\int_{0}^{M}f_{1}(T^{t}x)f_{2}(T^{at}x)dt
		\end{equation} exists almost everywhere.
\end{thm}

\begin{proof}
	By Theorem \ref{thm1}, we can assume that $(X,(T^{t})_{t\in \R})$ is a continuous flow, $\X$ is the Borel $\sigma$-algebra of $X$, and $\mu$ is a $(T^{t})_{t\in \R}$-invariant Borel probability measure on $X$. By Theorem
	\ref{thm2}, it suffices to prove that the limit in \eqref{eq26} exists almost everywhere for all $f_1,f_2\in C(X)$.
	
	Let $f_1,f_2\in C(X)$ such that $\norm{f_1}_{\infty}\le 1$ and $\norm{f_2}_{\infty}\le 1$. Then, for any $\ep>0$, there exists $\eta(\ep)\in (0,1)$ such that for any $x\in X$ and any $t\in (-\eta(\ep),\eta(\ep)$, we have  $$|f_{1}(T^{t}x)-f_{1}(x)|<\ep, |f_{2}(T^{t}x)-f_{2}(x)|<\ep.$$
	This property implies that for any $k\in \N$, there exists $\delta_{k}\in (0,1/k)$ such that for any $x\in X$,
	\begin{align}
		& \limsup_{N\to\infty}\left|\frac{1}{N\delta_k}\int_{0}^{N\delta_k}f_{1}(T^{t}x)f_{2}(T^{at}x)dt-\frac{1}{N}\sum_{n=0}^{N-1}f_{1}((T^{\delta_k})^{n}x)f_{2}((T^{\delta_k})^{\lfloor an\rfloor}x)\right|\le \frac{3}{k} \label{eq2-1}
	\end{align}
	By combining Theorem \ref{thm-K}, \eqref{eq2-1} and the fact that $\delta_k\to 0$ as $k\to\infty$, we conclude that the limit $$	\lim\limits_{M\to\infty}\frac{1}{M}\int_{0}^{M}f_{1}(T^{t}x)f_{2}(T^{at}x)dt$$ exists almost everywhere. The proof is complete.
\end{proof}	
In \cite{FN}, Frantzikinakis built a pointwise ergodic theorem of measurable flows without any commutativity.
To prove (1) of Theorem D, we need the following result, which is a special case of \cite[Theorem 1.11]{FN}.

\begin{thm}\label{thm4}
Let $(X,\X,\mu, (T_{i}^{t})_{t\in \R}),1\le i\le k$ be measurable flows, where $k\in \N$. Then for any distinct positive real numbers $c_1,\ldots,c_k$ and any $f_1, \ldots, f_k\in L^{\infty}(\mu)$,
	  	\begin{equation}\label{eq3}
	  		\lim\limits_{M\to\infty}\frac{1}{M}\int_{0}^{M}f_{1}(T_{1}^{t^{c_1}}x)\cdots f_{k}(T_{k}^{t^{c_k}}x)dt=\prod_{i=1}^{k}\E_{\mu}(f_{i}|\mathcal{I}((T_{i}^{t})_{t\in\R}))(x)
	  	\end{equation}
almost everywhere.
	  \end{thm}
	
	
	

\section{Proofs of Theorem B and Theorem C}\label{section-BC}
	
In this section, we present the proofs for Theorem B and Theorem C. Given the complexity of these proofs, we will provide a sketch of the proof for a special case of Theorem A to illustrate how our method works before presenting the detailed proofs of Theorem B and Theorem C.

\subsection{Two lemmas}

The following lemma is a special case of \cite[Lemma 7.2]{FN}.
\begin{lemma}\label{lem1}
	Let $\delta>0$. Let $f\in L^{\infty}(m_{\R})$, where $m_{\R}$ is the Lebesgue measure on $\R$. If the limit $\displaystyle \lim\limits_{M\to\infty}\frac{1}{M}\int_{0}^{M}f(t)dt$ exists,
	then the limit
	$$\lim\limits_{M\to\infty}\frac{1}{M}\int_{0}^{M}f(t^\delta)dt$$ exists and the two limits are equal.
\end{lemma}

The following lemma, which is a corollary of Theorem \ref{thm2}, will be used in the proofs of Theorem B and $(1)$ of Theorem D.

\begin{lemma}\label{lem4}
Let $(X,\X,\mu, (T^{t})_{{ t}\in \R})$ be a measurable flow and $\alpha\in (0,1]$. Then there exists a constant $C=C(\alpha)>0$ such that for any $f\in L^{\infty}(\mu)$, one has
\begin{equation}\label{eq23}
\norm{\sup_{M\in\R_+}\Big|\frac{1}{M}\int_{0}^{M}f(T^{t^{\alpha}}x)dt\Big|}_2\le C\norm{f}_2.
		\end{equation}
	\end{lemma}

\begin{proof}
For any $f\in L^\infty(\mu)$, we have
\begin{align*}
   & \Big|\frac{1}{M}\int_{0}^{M}f(T^{t^{\alpha}}x)dt\Big|  \overset{s=t^\a}= \Big| \frac{1}{M}\int_{0}^{M^\a}f(T^{s}x)\frac{1}{\a}s^{\frac{1}{\a}-1}ds\Big|
    \\= &   \frac{1}{\a}\frac{1}{M^\a}\Big|\int_{0}^{M^\a}f(T^{s}x)\Big(\frac{s}{M^\a}\Big)^{\frac{1}{\a}-1}ds\Big|
    \le \frac{1}{\a}\frac{1}{M^\a}\int_{0}^{M^\a}|f|(T^{s}x)ds,
\end{align*}
where $\Big(\frac{s}{M^\a}\Big)^{\frac{1}{\a}-1}\le 1$ as $\a\in (0,1]$.
Thus the result follows from Theorem \ref{thm2}.
\end{proof}
%

\subsection{A sketch of the proof of a special case}\label{sample}
In this subsection, we provide a sketch of the proof for a special case of Theorem A, illustrating the key ideas behind the proofs of the main results of the paper.

Now, we present the following special case.

\medskip
\noindent {\bf Special case.}\ {\em
	Let $(X,(T^t)_{t\in\R})$ and $(X,(S^t)_{t\in\R})$ be two continuous flows, and let $\mu$ be a Borel probability measure on $X$ that is invariant under both $(T^t)_{t\in\R}$ and $(S^t)_{t\in\R}$. Then for any $f,g\in L^{\infty}(\mu)$, the limit
	\begin{equation}\label{SEQ}
		\lim\limits_{M\to\infty}\frac{1}{M}\int_{0}^{M}f(T^{t}x)g(S^{t^2}x)dt
	\end{equation}
	exists almost everywhere.
}
\medskip

For our purpose, we will only provide a sketch the proof for this special case. For a complete proof, the reader is referred to the proof of Theorem B.
\begin{proof}
	By Theorem \ref{thm2}, it suffices to prove that the limit in \eqref{SEQ} exists almost everywhere for all $f,g\in C(X)$. Let $f,g\in C(X)$ such that for any $x\in X$, $|f(x)|\le 1$ and $|g(x)|\le 1$. Choose $\ep>0$ arbitrarily. Then there exists $\delta\in (0,1)$ such that for any $x\in X,t\in (-\delta,\delta)$, we have
	\begin{equation}\label{SEQ1}
		|f(T^{t}x)-f(x)|<\ep,\quad |g(S^{t}x)-g(x)|<\ep.
	\end{equation}
	By \eqref{SEQ1}, we have
	\begin{align}
		& \int_{X}\limsup_{K\to\infty}\sup_{M_1\ge K,\atop M_{2}\ge K}\Big|\frac{1}{M_1}\int_{0}^{M_1}f(T^{t}x)g(S^{t^2}x)dt-\frac{1}{M_2}\int_{0}^{M_2}f(T^{t}x)g(S^{t^2}x)dt\Big|d\mu(x) \notag
		\\ \le &
		 \int_{X}\limsup_{K\to\infty}\sup_{M_1\ge K,\atop M_{2}\ge K}\Big|\frac{1}{\lfloor M_1/\delta\rfloor\delta}\int_{0}^{\lfloor M_1/\delta\rfloor\delta}f(T^{t}x)g(S^{t^2}x)dt\notag\\ & \hspace{5cm}-\frac{1}{\lfloor M_2/\delta\rfloor\delta}\int_{0}^{\lfloor M_2/\delta\rfloor\delta}f(T^{t}x)g(S^{t^2}x)dt\Big|d\mu(x) \notag
		 \\ = &
		  \int_{X}\limsup_{K\to\infty}\sup_{M_1\ge K,\atop M_{2}\ge K}\Big|\E_{n\in I_{\lfloor M_1/\delta\rfloor}}\int_{0}^{1}f((T^{\delta})^{(n+t)}x)g((S^{
		  	\delta^2})^{(n+t)^2}x)dt\notag\\ & \hspace{5cm}-
		  \E_{n\in I_{\lfloor M_2/\delta\rfloor}}\int_{0}^{1}f((T^{\delta})^{(n+t)}x)g((S^{
		  	\delta^2})^{(n+t)^2}x)dt\Big|d\mu(x) \notag
		  \\ \overset{\eqref{SEQ1}}\le &
		   4\ep+\int_{X}\limsup_{K\to\infty}\sup_{M_1\ge K,\atop M_{2}\ge K}\Big|\E_{n\in I_{\lfloor M_1/\delta\rfloor}}\int_{0}^{1}f((T^{\delta})^{n}x)g((S^{
		   	\delta^2})^{n^{2}+2nt}x)dt\notag\\ & \hspace{5cm}-
		   \E_{n\in I_{\lfloor M_2/\delta\rfloor}}\int_{0}^{1}f((T^{\delta})^{n}x)g((S^{
		   	\delta^2})^{n^{2}+2nt}x)dt\Big|d\mu(x).\notag
	\end{align}
	
	Let $T_1=T^{\delta},S_1=S^{\delta^2}$. For any $x\in X$, let $\displaystyle F(x)=\int_{0}^{1}g(S_{1}^{t}x)dt$. Then we can write  $\displaystyle\int_{0}^{1}g(S_{1}^{n^{2}+2nt}x)dt$ as $\E_{i\in I_{2n}}F(S_{1}^{n^2+i}x)$. Therefore, it is left to show that
	\begin{equation}\label{SEQ3}
		\lim_{N\to\infty}\E_{n\in I_N}f(T_{1}^{n}x)\big(\E_{i\in I_{2n}}S_{1}^{i}F\big)(S_{1}^{n^2}x)
	\end{equation}
exists almost everywhere.
	
	For any $r,l\in \N$, let $$E_{r}^{l}=\{x\in X:\sup_{n\ge l}\Big|\E_{i\in I_{n}}F(S_{1}^{i}x)-\E_{\mu}(F|\mathcal{I}(S_1))(x)\Big|<\frac{1}{r}\}.$$
	Then by Birkhoff's ergodic theorem, for any $r\in \N$, there exists $l_r\in \N$ such that for any $l\ge l_r$, $\mu(E_{r}^{l})>1-\frac{1}{r}$.
	Fix $r\in \N$ arbitrarily. Note that for any $N\ge 100l_r$ and $\mu$-a.e. $x\in X$, we have
	\begin{align}
		& \Big|\E_{n\in I_N}f(T_{1}^{n}x)\Big(\E_{i\in I_{2n}}S_{1}^{i}F-\E_{\mu}(F|\mathcal{I}(S_1))\Big)(S_{1}^{n^2}x)\Big| \notag
		\\ \le &
		\frac{1}{N}\sum_{n=0}^{l_{r}} \Big| f(T_{1}^{n}x)\Big(\E_{i\in I_{2n}}S_{1}^{i}F-\E_{\mu}(F|\mathcal{I}(S_1))\Big)(S_{1}^{n^2}x)  \Big| \label{SEQ2}
		\\ & \hspace{2cm} +
		\frac{1}{N}\sum_{l_r<n<N,\atop S_{1}^{n^2}x\in E_{r}^{l_r}} \Big| f(T_{1}^{n}x)\Big(\E_{i\in I_{2n}}S_{1}^{i}F-\E_{\mu}(F|\mathcal{I}(S_1))\Big)(S_{1}^{n^2}x)  \Big| \notag
	\\ & \hspace{4cm} +
			\frac{1}{N}\sum_{l_r<n<N,\atop S_{1}^{n^2}x\notin E_{r}^{l_r}} \Big| f(T_{1}^{n}x)\Big(\E_{i\in I_{2n}}S_{1}^{i}F-\E_{\mu}(F|\mathcal{I}(S_1))\Big)(S_{1}^{n^2}x)  \Big| \notag
	\\ \le &	
	\frac{2(l_r+1)}{N}+ \frac{1}{r}+ 2\E_{n\in I_N}1_{(E_{r}^{l_r})^{c}}(S_{1}^{n^2}x). \notag
	\end{align}
	By \eqref{SEQ2} and Theorem \ref{BPET}, we have
	\begin{align*}
		&\int_{X}\limsup_{N\to\infty}\Big|\E_{n\in I_N}f(T_{1}^{n}x)\Big(\E_{i\in I_{2n}}S_{1}^{i}F-\E_{\mu}(F|\mathcal{I}(S_1))\Big)(S_{1}^{n^2}x)\Big|d\mu(x)
		\\ \le &
		\frac{1}{r}+ 2\int_{X}\limsup_{N\to\infty}\E_{n\in I_N}1_{(E_{r}^{l_r})^{c}}(S_{1}^{n^2}x)d\mu(x)
		\le
		\frac{3}{r}.
	\end{align*}
	Since $r$ is arbitrary, for $\mu$-a.e. $x\in X$,
	$$\lim_{N\to\infty}\Big|\E_{n\in I_N}f(T_{1}^{n}x)\big(\E_{i\in I_{2n}}S_{1}^{i}F\big)(S_{1}^{n^2}x)-\E_{n\in I_N}f(T_{1}^{n}x)\E_{\mu}(F|\mathcal{I}(S_1))(x)\Big|=0.$$
	It follows that the limit in \eqref{SEQ3} exists almost everywhere if and only if the limit $$\lim_{N\to\infty}\E_{n\in I_N}f(T_{1}^{n}x)$$ exists almost everywhere.
	
	 Therefore, Birkhoff's ergodic theorem implies that the limit in \eqref{SEQ3} exists almost everywhere. The proof is complete.
\end{proof}
In the above proof, the property of continuous functions can allow us to reduce the continuous-time ergodic averages to discrete-time ergodic averages, whose pointwise convergence can be deduced from known ergodic theorems.

In the following proofs, we will utilize the topological models of measurable flows and various ergodic theorems to apply the method outlined in the previous proof.

\subsection{Proof of Theorem B}

We rewrite (\ref{TA1}) as
$$\lim_{M\to\infty}\frac{1}{M}\int_{0}^{M}f_{1}(T^{(t^\beta)^{\alpha/\beta}}x)
f_{2}(T^{a(t^\beta)^{\alpha/\beta}}x)g(S^{Q(t^\beta)}x)dt.$$
By Lemma \ref{lem1}, we can assume that $\beta=1$. Moreover, we can suppose that $a\neq 0$, $Q(0)=0$ and the leading coefficient of $Q(t)$ is $1$. At this point, we need to prove the following results:
		
Let $\gamma\in (0,1]$. For any $f_1,f_2, g\in L^{\infty}(\mu)$, the limit
\begin{equation}\label{EB}
  \lim\limits_{M\to\infty}\frac{1}{M}\int_{0}^{M}f_{1}(T^{t^{\gamma}}x)
f_{2}(T^{at^{\gamma}}x)g(S^{Q(t)}x)dt
\end{equation}
exists almost everywhere. And if we denote the limit function by $L(f_1,f_2,g)$, then
\begin{equation}\label{LB}	L(f_1,f_2,g)=\E_{\mu}(g|\mathcal{I}((S^t)_{t\in\R}))(x)\lim\limits_{M\to\infty}\frac{1}{M}\int_{0}^{M}f_{1}(T^{t}x)
	f_{2}(T^{at}x)dt
\end{equation}
almost everywhere.
	
The rest of the proof is divided into three steps.
		
\noindent {\textbf{Step I. Reduction to two dense subsets of $L^{2}(\mu)$.}}
		
By Theorem \ref{thm1}, there are continuous flows $(Y,(R^t)_{t\in\R})$ and $(Z,(H^t)_{t\in\R})$, $(R^t)_{t\in\R}$-invariant Borel probability measure $\nu$ on $Y$, and $(H^t)_{t\in\R}$-invariant Borel probability measure $\tau$ on $Z$, such that $(X,\X,\mu, (T^{t})_{t\in \R})$ and $(Y,\B(Y),\nu,(R^t)_{t\in\R})$ are isomorphic, and $(X,\X,\mu, (S^{t})_{t\in \R})$ and $(Z,\mathcal{B}(Z),\tau,(H^t)_{t\in\R})$ are also isomorphic. Thus there is an invertible measure preserving map $\Phi$ between $(T^t)_{t\in\R}$-invariant measurable subset $X_0\subset X$ and $(R^t)_{t\in\R}$-invariant measurable subset $Y_0\subset Y$, both of measure $1$, such that $$\Phi\circ T^{t}(x)=R^{t}\circ \Phi(x),\ \ \text{for}\ \text{each}\  t\in \R, x\in X_0.$$ Similarly, there is an invertible measure preserving map $\Psi$ between $(S^t)_{t\in\R}$-invariant measurable subsets $X_1\subset X$ and  $(H^t)_{t\in\R}$-invariant measurable subsets $Z_1\subset Z$, both of measure $1$, such that $$\Psi\circ S^{t}(x)=H^{t}\circ \Psi(x),\ \text{for}\ \text{each}\  t\in \R, x\in X_1.$$ The isomorphisms $\Phi$ and $\Psi$ induce isomorphisms of Hilbert spaces:
$$\Phi_*: L^2(Y,\nu)\rightarrow L^2(X,\mu), h\mapsto h\circ \Phi,\quad \Psi_*: L^2(Z,\tau)\rightarrow L^2(X,\mu), h\mapsto h\circ \Psi.$$
Let $$\mathcal{A}((T^t)_{t\in \R})=\Phi_*\big(C(Y)\big)=\{h_1\circ\Phi:h_1\in C(Y)\},$$ and let
$$\mathcal{A}((S^t)_{t\in \R})=\Psi_*\big(C(Z)\big)=\{h_2\circ\Psi:h_2\in C(Z)\}.$$

As $C(Y)$ is dense in $L^2(Y,\nu)$ and $C(Z)$ is dense in $L^2(Z,\tau)$, $\mathcal{A}((T^t)_{t\in \R})$ and $\mathcal{A}((S^t)_{t\in \R})$ are two dense subsets of $L^{2}(X, \mu)$. From the definitions of $\mathcal{A}((T^t)_{t\in \R})$ and $\mathcal{A}((S^t)_{t\in \R})$, we observe the following:

\noindent  (\textbf{S1}) For any $\tilde{f}\in\mathcal{A}((T^t)_{t\in \R})$, and any $\tilde{g}\in\mathcal{A}((S^t)_{t\in \R})$, $$|\tilde{f}(x_0)|\le \norm{\tilde{f}}_{\infty}\ \text{for}\ \text{each}\  x_0\in X_0; \ \text{and}\ |\tilde{g}(x_1)|\le \norm{\tilde{g}}_{\infty}\ \text{for}\ \text{each}\  x_1\in X_1.$$

Moreover, the following statement holds:
		
\noindent (\textbf{S2}) For any $\epsilon>0$, any $\tilde{f}\in\mathcal{A}((T^t)_{t\in \R})$, and any $\tilde{g}\in\mathcal{A}((S^t)_{t\in \R})$, there is $\eta>0$ such that for any $x_0\in X_0$,  $x_1\in X_1$, and any $t\in (-\eta,\eta)$,  $$|\tilde{f}(T^{t}x_0)-\tilde{f}(x_0)|<\epsilon,\ \text{and}\  |\tilde{g}(S^{t}x_1)-\tilde{g}(x_1)|<\epsilon.$$

To see this, let $\tilde{f}=f\circ \Phi$ and $\tilde{g}=g\circ \Psi$ for some $f\in C(Y)$ and $g\in C(Z)$. As $(Y,(R^t)_{t\in\R})$ and $(Z,(H^t)_{t\in\R})$ are continuous flows, there is $\eta>0$ such that for all $|t|<\eta,y\in Y$ and $z\in Z$,
$$|f(R^ty)-f(y)|<\epsilon, \quad|g(H^tz)-g(z)|<\ep.$$
Then for each $x_0\in X_0$ and each $t\in (-\eta, \eta)$, we have $\Phi (x_0)\in Y$ and
$$|\tilde{f}(T^tx_0)-\tilde{f}(x_0)|=|f(\Phi(T^tx_0))-f(\Phi(x_0))|=|f(R^t(\Phi (x_0)))-f(\Phi (x_0))|<\ep.$$
Similarly, for each $x_1\in X_1$ and each $t\in (-\eta, \eta)$, we have $|\tilde{g}(S^{t}x_1)-\tilde{g}(x_1)|<\epsilon.$ Thus we have established (\textbf{S2}).

By Theorem \ref{thm2} and Lemma \ref{lem4}, to verify that the limit in \eqref{EB} exists almost everywhere, it suffices to prove that the limit in \eqref{EB} exists almost everywhere for all $f_1,f_2\in \mathcal{A}((T^t)_{t\in \R})$ and $g\in \mathcal{A}((S^t)_{t\in \R})$.

\medskip	
\noindent 	\textbf{Step II. Verify that the limit in \eqref{EB} exists almost everywhere.}
			
Let $f_1,f_2\in\mathcal{A}((T^t)_{t\in \R})$ and $f_3\in\mathcal{A}((S^t)_{t\in \R})$ such that $\norm{f_1}_{\infty}\le 1$, $\norm{f_2}_{\infty}\le 1$, and $\norm{f_3}_{\infty}\le 1$. Next, we verify that the limit
\begin{equation}\label{EQ}
	  \lim\limits_{M\to\infty}\frac{1}{M}\int_{0}^{M}f_{1}(T^{t^{\gamma}}x)
	  f_{2}(T^{at^{\gamma}}x)f_{3}(S^{Q(t)}x)dt
\end{equation} exists for $\mu$-a.e. $x\in X$.

By (\textbf{S2}), for any $\theta>0$, there exists $\eta(\theta)\in (0,1)$ such that for any $x_0\in X_0$, any $x_1\in X_1$ and any $t\in (-\eta(\theta),\eta(\theta))$,
			\begin{equation}\label{eq6}
				|f_1(T^{t}x_0)-f_1(x_0)|<\theta,|f_2(T^{t}x_0)-f_2(x_0)|<\theta,\ \text{and}\ |f_3(S^{t}x_1)-f_3(x_1)|<\theta.
			\end{equation}

Choose $\ep>0$ arbitrarily. Let $\delta>0$ such that $0<\delta<\eta(\ep)/4, |a\delta|<\eta(\ep)$ and  $Q([0,\delta])\subset (-\eta(\ep),\eta(\ep))$.

Since $Q\in \R[t]$ with $\deg Q=s\ge 2,Q(0)=0$ and the leading coefficient of  $Q(t)$ is $1$, there are $P(n), P_1(n),\ldots, P_{s-1}(n)\in \R[n]$ such that following hold:
\begin{itemize}
	\item $P(0)=0$.
	\item For each $i\in \{1,\ldots,s-1\}$, $\deg P_{i}(n)=s-i$.
	\item The leading coefficient of  $P_{1}(n)$ is $s$.
	\item For any $n\in\Z,t\in\R$,
	\begin{equation}\label{eq51}
	Q(n\delta+t)=P(n)\delta^{s}+Q(t)+\sum_{i=1}^{s-1}P_{i}(n)\delta^{s-i}t^{i}.
	\end{equation}
\end{itemize}
%

Let $T_{1}=T^{\delta^{\gamma}},S_{1}=S^{\delta^s}$. We aim to prove that if the limit
$$\lim\limits_{N\to\infty}\E_{n\in I_N}f_{1}(T_{1}^{n^{\gamma}}x)f_{2}(T_{1}^{an^{\gamma}}x)$$ exists almost everywhere, then the limit in \eqref{EQ} also exists almost everywhere.

Now, we need the following ergodic averages for an approximation argument:
\begin{equation}\label{eq0}
		\E_{n\in I_N}f_{1}(T_{1}^{n^{\gamma}}x)f_{2}(T_{1}^{an^{\gamma}}x)\int_{0}^{1}f_{3}(S_{1}^{P(n)+\sum_{i=1}^{s-1}P_{i}(n)t^{i}}x)dt.
\end{equation}

For any $x\in X_0\cap X_1$ and any $N\in \N$, we have
\begin{align}
		& \frac{1}{N\delta}\int_{0}^{N\delta}f_{1}(T^{t^{\gamma}}x)f_{2}(T^{at^{\gamma}}x)f_{3}(S^{Q(t)}x)dt \notag
		\\\overset{\eqref{eq51}}= &
		\E_{n\in I_N}\int_{0}^{1}f_{1}((T^{\delta^{\gamma}})^{(n+t)^{\gamma}}x)f_{2}((T^{\delta^{\gamma}})^{a(n+t)^{\gamma}}x)f_{3}(S^{P(n)\delta^{s}+Q(\delta  t)+\sum_{i=1}^{s-1}P_{i}(n)\delta^{s}t^{i}}x)dt.\label{eq50}
\end{align}

By Mean Value Theorem, it is easy to verify that $(n+t)^\gamma-n^\gamma\le{\gamma}/{n^{1-\gamma}}$ for all $t\in [0,1],n\in\N$.
Thus there exists $N_0\in\N$ such that for any $n\ge N_0$ and any $t\in [0,1]$, $(n+t)^{\gamma}-n^{\gamma}\le\delta^{1-\gamma}$.

By \eqref{eq50}, the defintions of $T_1,S_1$, and the setting of $\delta$, we have that for any $x\in X_0\cap X_1$ and any $N\ge 100N_0$,
\begin{align}
& \Big|\frac{1}{N\delta}\int_{0}^{N\delta}f_{1}(T^{t^{\gamma}}x)f_{2}(T^{at^{\gamma}}x)f_{3}(S^{Q(t)}x)dt \notag
			\\ & \hspace{2cm} - \E_{n\in I_N}f_{1}(T_{1}^{n^{\gamma}}x)f_{2}(T_{1}^{an^{\gamma}}x)\int_{0}^{1}f_{3}(S_{1}^{P(n)+\sum_{i=1}^{s-1}P_{i}(n)t^{i}}x)dt\Big| \notag
\\\overset{\eqref{eq50},(\textbf{S1}),(\ref{eq6})}\le &
			 3\epsilon+\frac{2N_0}{N}. \label{eq11}
\end{align}

			Let $G_{n}(t)=\sum_{i=1}^{s-1}P_{i}(n)t^{i}$ for any $n\in \Z$. Since the leading coefficient of  $P_{1}(n)$ is $s$ and $\deg P_{i}(n)=s-i$ for each $i\in \{1,\ldots,s-1\}$, there exists $N_1\in \N$ with $N_1>100N_0$ such that for any $n\ge N_1$ and any $t\in [-1,2]$, we have that $G_{n}'(t)>0,\lfloor G_{n}(1)\rfloor >n$ and $P_{1}(n)>\sum_{i=2}^{s-1}|P_{i}(n)|+2$. (When $s=2$, we view $\sum_{i=2}^{s-1}|P_{i}(n)|$ as $0$.) So, for any $n\ge N_1$ and any $t\in [0,1]$, $G_{n}(t)$ has an inverse function, denoted by $G_{n}^{-1}$. Note that for any $x\in X_0\cap X_1$,
\begin{align}
& \limsup_{N\to\infty}\Big|	\E_{n\in I_N}f_{1}(T_{1}^{n^{\gamma}}x)f_{2}(T_{1}^{an^{\gamma}}x)\int_{0}^{1}f_{3}(S_{1}^{P(n)+G_{n}(t)}x)dt \notag
\\ & \hspace{1.5cm} -	\frac{N-N_1}{N}\E_{n\in I_N\backslash I_{N_1}}f_{1}(T_{1}^{n^{\gamma}}x)f_{2}(T_{1}^{an^{\gamma}}x)\frac{1}{\lfloor G_{n}(1)\rfloor}\int_{0}^{\lfloor G_{n}(1)\rfloor} f_{3}(S_{1}^{P(n)+t}x)dt\Big|  \notag
				\\ \overset{(\textbf{S1})}\le  &
				\limsup_{N\to\infty}\E_{n\in I_N\backslash I_{N_1}}\Big|\int_{0}^{G_{n}(1)}\Big( \frac{1}{G_{n}'(G_{n}^{-1}(t))}-\frac{1}{G_{n}(1)}\Big)f_{3}(S_{1}^{P(n)+t}x)dt \Big|  \notag
				\\ \overset{(\textbf{S1})}\le &
				\limsup_{N\to\infty}\E_{n\in I_N\backslash I_{N_1}}\frac{1+s\sum_{i=2}^{s-1}|P_{i}(n)|}{P_{1}(n)-(s-1)\sum_{i=2}^{s-1}|P_{i}(n)|}  \label{eq24}
			=
			0.
\end{align}

For any $x\in X_0\cap X_1$, let $\displaystyle F(x)=\int_{0}^{1}f_{3}(S_{1}^{t}x)dt$. By the definition of $S_1$ and the choice of $\delta$, we have that for any $x\in X_0\cap X_1$ and any $N\ge 100N_1$,
\begin{equation}\label{eq41}
 \begin{split}
 & \Big|\E_{n\in I_N\backslash I_{N_1}}f_{1}(T_{1}^{n^{\gamma}}x)f_{2}(T_{1}^{an^{\gamma}}x)\frac{1}{\lfloor G_{n}(1)\rfloor}\int_{0}^{\lfloor G_{n}(1)\rfloor}f_{3}(S_{1}^{P(n)+t}x)dt
 \\ & \hspace{1cm} -
 \E_{n\in I_N\backslash I_{N_1}}f_{1}(T_{1}^{n^{\gamma}}x)f_{2}(T_{1}^{an^{\gamma}}x)\Big(\E_{i\in I_{\lfloor G_{n}(1)\rfloor}}S_{1}^{i} F\Big)(S_{1}^{\lfloor P(n)\rfloor}x)\Big|\le \ep
 \end{split}
\end{equation}

For any $r, l\in\N$, let
			$$E_{r}^{l}=\Big\{x\in X:\sup_{n\ge l}\Big|\E_{i\in I_n}F(S_{1}^{i}x)-\E_{\mu}(F|\mathcal{I}(S_1))(x)\Big|<\frac{1}{r}\Big\}.$$
Fix $r\in \N$ arbitrarily. By Birkhoff's ergodic theorem, there is $l_r\in\N$ with $l_r\ge 100N_1$ such that for all $l\ge l_r$, $\mu(E_{r}^{l})>1-\frac{1}{r}$. Note that for $\mu$-a.e. $x\in X_0\cap X_1$ and any $N \ge 100l_r$,
\begin{align}
&\Big|\E_{n\in I_N\backslash I_{N_1}}f_{1}(T_{1}^{n^{\gamma}}x)
f_{2}(T_{1}^{an^{\gamma}}x)\Big(\big(\E_{i\in I_{\lfloor G_{n}(1)\rfloor}}S_{1}^{i}F\big)-
\E_{\mu}(F|\mathcal{I}(S_1))\Big)(S_{1}^{\lfloor P(n)\rfloor}x)\Big| \notag
\\ \le &
				\frac{2}{N}\sum_{n=N_1}^{l_r}\Big| f_{1}(T_{1}^{n^{\gamma}}x)f_{2}(T_{1}^{an^{\gamma}}x)\Big(\big(\E_{i\in I_{\lfloor G_{n}(1)\rfloor}}S_{1}^{i}F\big)-
				\E_{\mu}(F|\mathcal{I}(S_1))\Big)(S_{1}^{\lfloor P(n)\rfloor}x)\Big|  \label{eq52}\\ & +\frac{2}{N}\sum_{l_r< n\le N-1, \atop S_{1}^{\lfloor P(n)\rfloor}x\in E_{r}^{l_r}}\Big| f_{1}(T_{1}^{n^{\gamma}}x)f_{2}(T_{1}^{an^{\gamma}}x)\Big(\big(\E_{i\in I_{\lfloor G_{n}(1)\rfloor}}S_{1}^{i}F\big)-
				\E_{\mu}(F|\mathcal{I}(S_1))\Big)(S_{1}^{\lfloor P(n)\rfloor}x)\Big| \notag\\ & + \frac{2}{N}\sum_{l_r< n\le N-1, \atop S_{1}^{\lfloor P(n)\rfloor}x\notin E_{r}^{l_r}}\Big| f_{1}(T_{1}^{n^{\gamma}}x)f_{2}(T_{1}^{an^{\gamma}}x)\Big(\big(\E_{i\in I_{\lfloor G_{n}(1)\rfloor}}S_{1}^{i}F\big)-
				\E_{\mu}(F|\mathcal{I}(S_1))\Big)(S_{1}^{\lfloor P(n)\rfloor}x)\Big| \notag
				\\ \overset{(\textbf{S1})}\le &
				\frac{4l_r}{N}+\frac{2}{r}+
				4\E_{n\in I_N}1_{(E_{r}^{l_r})^{c}}(S_{1}^{\lfloor P(n)\rfloor}x). \notag
	\end{align}
By \eqref{eq52} and Theorem \ref{BPET}, we have
\begin{align*}
& \int \limsup_{N\to\infty} \Big|\E_{n\in I_N\backslash I_{N_1}}f_{1}(T_{1}^{n^{\gamma}}x)f_{2}(T_{1}^{an^{\gamma}}x) \\ & \hspace{3cm} \cdot \Big(\big(\E_{i\in I_{\lfloor G_{n}(1)\rfloor}}S_{1}^{i}F\big)-
\E_{\mu}(F|\mathcal{I}(S_1))\Big)(S_{1}^{\lfloor P(n)\rfloor}x)\Big|d\mu(x)
		    	\\ \le &
		    	 \frac{2}{r}+ 4\int \limsup_{N\to\infty}\E_{n\in I_N}1_{(E_{r}^{l_r})^{c}}(S_{1}^{\lfloor P(n)\rfloor}x)d\mu(x)
		    	\le
		    	\frac{6}{r}.
		    \end{align*}
Since $r$ is arbitrary, we have
\begin{equation}\label{eq20}
	\lim_{N\to\infty}\Big|\E_{n\in I_N\backslash I_{N_0}}f_{1}(T_{1}^{n^{\gamma}}x)
	f_{2}(T_{1}^{an^{\gamma}}x)\Big(\big(\E_{i\in I_{\lfloor G_{n}(1)\rfloor}}S_{1}^{i}F\big)-
	\E_{\mu}(F|\mathcal{I}(S_1))\Big)(S_{1}^{\lfloor P(n)\rfloor}x)\Big|=0
\end{equation}
for $\mu$-a.e. $x\in X_0\cap X_1$.

For ease of reference, we refer to the process outlined above, which constructs (\ref{eq20}), as the \textbf{Reduction Process}.

By \eqref{eq11}, \eqref{eq24}, \eqref{eq41}, and \eqref{eq20}, we have
\begin{align*}
&\int_X \limsup_{K\to\infty}\sup_{M_1\ge K,\atop M_2\ge K}\Big|\frac{1}{M_1}\int_{0}^{M_1}f_{1}(T^{t^{\gamma}}x)f_{2}(T^{at^{\gamma}}x)f_{3}(S^{Q(t)}x)dt
\\ & \hspace{3cm}- \frac{1}{M_2}\int_{0}^{M_2}f_{1}(T^{t^{\gamma}}x)f_{2}(T^{at^{\gamma}}x)f_{3}(S^{Q(t)}x)dt\Big|d\mu(x)
\\\le  & \int_X \limsup_{K\to\infty}\sup_{M_1\ge K,\atop M_2\ge K}\Big|\frac{1}{\lfloor M_1/\delta \rfloor\delta}\int_{0}^{\lfloor M_1/\delta \rfloor\delta}f_{1}(T^{t^{\gamma}}x)f_{2}(T^{at^{\gamma}}x)f_{3}(S^{Q(t)}x)dt
\\ & \hspace{3cm}- \frac{1}{\lfloor M_2/\delta \rfloor\delta}\int_{0}^{\lfloor M_2/\delta \rfloor\delta}f_{1}(T^{t^{\gamma}}x)f_{2}(T^{at^{\gamma}}x)f_{3}(S^{Q(t)}x)dt\Big|d\mu(x)
\\ \overset{{(\ref{eq11}),(\ref{eq24}),}\atop \eqref{eq41},\eqref{eq20}}\le &
\int_X \limsup_{K\to\infty}\sup_{M_1\ge K,\atop M_2\ge K}\Big|	\E_{n\in I_{\lfloor M_1/\delta \rfloor }}f_{1}(T_{1}^{n^{\gamma}}x)f_{2}(T_{1}^{an^{\gamma}}x)\E_{\mu}(F|\mathcal{I}(S_1))(x)
\\ & \hspace{1cm}- \E_{n\in I_{\lfloor M_2/\delta \rfloor }}f_{1}(T_{1}^{n^{\gamma}}x)f_{2}(T_{1}^{an^{\gamma}}x)\E_{\mu}(F|\mathcal{I}(S_1))(x)\Big|d\mu(x) + 6\epsilon.
\end{align*}

For ease of reference, we refer to the calculation outlined above as the \textbf{Reduction Calculation}.

By the above calculations, to verify that the limit in (\ref{EQ}) exists for $\mu$-a.e. $x\in X$, it suffices to show that
\begin{equation}\label{eq25} \lim\limits_{N\to\infty}\E_{n\in I_N}f_{1}(T_{1}^{n^{\gamma}}x)f_{2}(T_{1}^{an^{\gamma}}x)
\end{equation}
exists for $\mu$-a.e. $x\in X$.

 First, we consider the case where $\gamma=1$. By (\textbf{S1}) and \eqref{eq6},  for any $k\in \N$, there exists $L_k\in \N$ such that for any $x\in X_0$,
 $$\left|\E_{n\in I_N}f_{1}(T_{1}^{n}x)f_{2}(T_{1}^{an}x)-\E_{n\in I_N}f_{1}((T^{\delta/L_k})^{L_{k}n}x)f_{2}((T^{\delta/L_k})^{\lfloor (aL_{k})n\rfloor}x)\right|\le \frac{1}{k}.$$
 By combining Theorem \ref{thm-K} and the above fact, we have that
 the limit in \eqref{eq25} exists for $\mu$-a.e. $x\in X$.

Next, we consider the case $\gamma\in (0,1)$. Choose $\tau>0$ arbitrarily.
By Theorem \ref{thm6} and Lemma \ref{lem1},  the limit
\begin{equation}\label{eq59}
  \lim\limits_{M\to\infty}\frac{1}{M}\int_{0}^{M}f_{1}(T_{1}^{t^{\gamma}}x)
f_{2}(T_{1}^{at^{\gamma}}x)dt
\end{equation}
exists almost everywhere.

As mentioned before, by Mean Value Theorem, we have that $(n+t)^\gamma-n^\gamma\le {\gamma}/{n^{1-\gamma}}$ for any $t\in [0,1],n\in\N$. Thus there exists $N_2\in \N$ such that for any $n\ge N_2$ and any $t\in [0,1]$, $(n+t)^{\gamma}-n^{\gamma}<\min\{1,|1/a|\}\eta(\tau)/\delta^{\gamma}$.

 Then for any $x\in X_0\cap X_1$,
\begin{align}
& \limsup_{M\to\infty} \Big| \frac{1}{M}\int_{0}^{M}f_{1}(T_{1}^{t^{\gamma}}x) f_{2}(T_{1}^{at^{\gamma}}x)dt- \E_{n\in I_{\lfloor M\rfloor}}f_{1}(T_{1}^{n^{\gamma}}x) f_{2}(T_{1}^{an^{\gamma}}x) \Big| \notag
\\\le & \limsup_{M\to\infty}\Big|\E_{n\in I_{\lfloor M\rfloor}}\int_{0}^{1}f_{1}(T_{1}^{(n+t)^{\gamma}}x)f_{2}(T_{1}^{a(n+t)^{\gamma}}x)dt-\E_{n\in I_{\lfloor M\rfloor}}f_{1}(T_{1}^{n^{\gamma}}x)f_{2}(T_{1}^{an^{\gamma}}x)\Big|\notag
		    	\\ \overset{(\textbf{S1}), (\ref{eq6})} \le &
		    	2\tau.   	\label{eq72}
		    \end{align}
		    Since $\tau$ is arbitrary and the limit in \eqref{eq59} exists almost everywhere, the limit in (\ref{eq25}) exists for $\mu$-a.e. $x\in X_0\cap X_1$. Therefore, the limit in (\ref{EQ}) exists for $\mu$-a.e. $x\in X$.
		
\noindent \textbf{Step III. Verify that the equality in \eqref{LB} holds almost everywhere.}

By Theorem \ref{thm2} and Lemma \ref{lem4}, it suffices to prove that the equality in \eqref{LB} holds almost everywhere for all $f_1,f_2\in \mathcal{A}((T^t)_{t\in \R})$ and $g\in \mathcal{A}((S^t)_{t\in \R})$.

Let $f_1,f_2\in\mathcal{A}((T^t)_{t\in \R})$ and $f_3\in\mathcal{A}((S^t)_{t\in \R})$ $\norm{f_1}_{\infty}\le 1$, $\norm{f_2}_{\infty}\le 1$, and $\norm{f_3}_{\infty}\le 1$. Next, we verify that
\begin{equation}\label{eq67}
L(f_1,f_2,f_3)=\E_{\mu}(f_3|\mathcal{I}((S^t)_{t\in\R}))(x)\lim\limits_{M\to\infty}\frac{1}{M}\int_{0}^{M}f_{1}(T^{t}x)f_{2}(T^{at}x)dt
\end{equation}
almost everywhere.

By (\textbf{S2}), for any $\epsilon>0$, there exist $\delta,\eta>0$ with $0<\delta<\eta/4, |a\delta|<\eta$ and $Q([0,\delta])\subset (-\eta,\eta)$ such that for any $x_0\in X_0$, any $x_1\in X_1$ and any $t\in (-\eta,\eta)$, we have
$$
	|f_1(T^{t}x_0)-f_1(x_0)|<\ep,|f_2(T^{t}x_0)-f_2(x_0)|<\ep,|f_3(S^{t}x_1)-f_3(x_1)|<\ep.
$$
Let $\hat{T}=T^{\delta^{\gamma}/q},\hat{S}=S^{\delta^{s}}$. For any $x\in X_0\cap X_1$, let $\hat{F}(x)=\displaystyle \int_{0}^{1}f_{3}(\hat{S}^{t}x)dt$.  By repeating the related arguments in \textbf{Step II}, we can establish the corresponding equations \eqref{eq11}, \eqref{eq24}, \eqref{eq41}, and \eqref{eq20} for $f_1,f_2,f_3,\ep,\delta,\hat{T}, \hat{S}$ and $\hat{F}$.
Then we have
\begin{align*}
	&\int_{X}\limsup_{M\to\infty}\Big|\frac{1}{M}\int_{0}^{M}f_{1}(T^{t^{\gamma}}x)f_{2}(T^{at^{\gamma}}x)f_{3}(S^{Q(t)}x)dt \\ & \hspace{3cm} -
	\E_{\mu}(\hat{F}|\mathcal{I}(\hat{S}))(x)\E_{n\in I_{\lfloor M/\delta \rfloor }}f_{1}(\hat{T}^{n^{\gamma}}x)f_{2}(\hat{T}^{an^{\gamma}}x)
	\Big|d\mu(x)
	\\\le &
	\int_{X}\limsup_{M\to\infty}\Big|\frac{1}{\lfloor M/\delta\rfloor\delta}\int_{0}^{\lfloor M/\delta\rfloor\delta}f_{1}(T^{t^{\gamma}}x)f_{2}(T^{at^{\gamma}}x)f_{3}(S^{Q(t)}x)dt \\ & \hspace{3cm} -
\E_{\mu}(\hat{F}|\mathcal{I}(\hat{S}))(x)\E_{n\in I_{\lfloor M/\delta \rfloor }}f_{1}(\hat{T}^{n^{\gamma}}x)f_{2}(\hat{T}^{an^{\gamma}}x)
	\Big|d\mu(x)
	\\ \overset{{\eqref{eq11}, \eqref{eq24}, }\atop\eqref{eq41},\eqref{eq20}}\le &
	32\ep.
\end{align*}

Note that the limit in \eqref{eq25} exists almost everywhere. By the above calculations, for any $k\in \N$, we can select proper positive real number $\delta_k$ such that
\begin{align}
	&\int_{X}\limsup_{M\to\infty}\Big|\frac{1}{M}\int_{0}^{M}f_{1}(T^{t^{\gamma}}x)f_{2}(T^{at^{\gamma}}x)f_{3}(S^{Q(t)}x)dt \label{eq65}\\ & \hspace{3cm} -
\E_{\mu}(F_k|\mathcal{I}(S_k))(x)\E_{n\in I_{\lfloor M\rfloor}}f_{1}(T_{k}^{n^{\gamma}}x)f_{2}(T_{k}^{an^{\gamma}}x)
	\Big|d\mu(x)<\frac{1}{4^k}, \notag
\end{align}
where $T_k=T^{\delta_{k}^{\gamma}},S_k=S^{\delta_{k}^{s}}$ and for any $x\in X_0\cap X_1$, $F_{k}(x)=\displaystyle \int_{0}^{1}f_{3}(S_{k}^{t}x)dt$. Moreover, we require that the sequence $\{\delta_k\}_{k\in\N}$ satisfies the following:
\begin{itemize}
	\item $\lim_{k\to\infty}\delta_k=0$.
	\item For any $k\in\N$, $\delta_k>\delta_{k+1}$ and $\delta_k/\delta_{k+1}\in \N$.
	\item For any $k\in \N$ and any $x\in X_0\cap X_1$,
	 $|F_{k}(x)-F_{k+1}(x)|<\frac{1}{4^{k}}$.
	 \item For any $k\in\N$, any $x\in X_0$ and any $t\in [0,1]$, we have
	 \begin{equation}\label{eq104}
	 |f_{1}(T^{\delta_{k}t}x)-f_{1}(x)|<\frac{1}{8^{k}(\norm{f_2}_{\infty}+1)},\quad |f_{2}(T^{a\delta_{k}t}x)-f_{2}(x)|<\frac{1}{8^{k}(\norm{f_1}_{\infty}+1)}.
	 \end{equation}
\end{itemize}
Clearly, there exists $F\in L^{\infty}(\mu)$ such that
\begin{equation}\label{eq68}
	\lim_{k\to\infty}F_k(x)=F(x)
\end{equation}
uniformly on $X_0\cap X_1$. Let $\displaystyle\mathcal{I}(S_{\infty})=\bigcap_{k\ge 1}\mathcal{I}(S_k)$. Then by Martingle Theorem and the fact that for any $k\in \N$, $\mathcal{I}(S_k)\supset \mathcal{I}(S_{k+1})$, we have
\begin{equation}\label{eq69}
	\lim_{k\to\infty}\E_{\mu}(F|\mathcal{I}(S_k))(x)=\E_{\mu}(F|\mathcal{I}(S_{\infty}))(x)
\end{equation} almost everywhere.

When $\gamma\in(0,1)$, by \eqref{eq59} and \eqref{eq72}, for $\mu$-a.e. $x\in X$ and any $k\in\N$,
\begin{equation}\label{eq71}
	\begin{split}
		\lim_{M\to\infty}\frac{1}{M}\int_{0}^{M}f_{1}(T_{k}^{t^{\gamma}}x)f_{2}(T_{k}^{at^{\gamma}}x)dt=
		\lim_{M\to\infty}\E_{n\in I_{\lfloor M\rfloor}}f_{1}(T_{k}^{n^{\gamma}}x)f_{2}(T_{k}^{an^{\gamma}}x).
	\end{split}
\end{equation}

When $\gamma=1$, by \eqref{eq104}, for any $k\in\N$,
\begin{equation}\label{eq105}
	\int_{X}\limsup_{M\to\infty}\Big|\E_{n\in I_{\lfloor M\rfloor}}f_{1}(T_{k}^{n}x)f_{2}(T_{k}^{an}x)-\frac{1}{M}\int_{0}^{M}f_{1}(T_{k}^{t}x)f_{2}(T_{k}^{at}x)dt
	\Big|d\mu(x)<\frac{1}{4^k}.
\end{equation}

For any $k\in \N$, we have
\begin{align}
	& \int_{X}\limsup_{M\to\infty}\Big|\frac{1}{M}\int_{0}^{M}f_{1}(T^{t^{\gamma}}x)f_{2}(T^{at^{\gamma}}x)f_{3}(S^{Q(t)}x)dt \notag\\ & \hspace{3cm} -
\E_{\mu}(F|\mathcal{I}(S_{\infty}))(x)\frac{1}{M}\int_{0}^{M}f_{1}(T_{k}^{t^{\gamma}}x)f_{2}(T_{k}^{at^{\gamma}}x)dt\Big|d\mu(x) \notag
	\\ \le &
	\int_{X}\limsup_{M\to\infty}\Big|\frac{1}{M}\int_{0}^{M}f_{1}(T^{t^{\gamma}}x)f_{2}(T^{at^{\gamma}}x)f_{3}(S^{Q(t)}x)dt \label{eq73}\\ & \hspace{3cm} -
\E_{\mu}(F_k|\mathcal{I}(S_k))(x)\E_{n\in I_{\lfloor M\rfloor}}f_{1}(T_{k}^{n^{\gamma}}x)f_{2}(T_{k}^{an^{\gamma}}x)
	\Big|d\mu(x)\notag
	\\ & +
	\int_{X}\limsup_{M\to\infty}\Big|\Big(\E_{\mu}(F|\mathcal{I}(S_k))-\E_{\mu}(F|\mathcal{I}(S_{\infty}))\Big)(x)\notag\\ & \hspace{3cm} \times\E_{n\in I_{\lfloor M\rfloor}}f_{1}(T_{k}^{n^{\gamma}}x)f_{2}(T_{k}^{an^{\gamma}}x)\Big|d\mu(x)\notag
	\\  & +
	\int_{X}\limsup_{M\to\infty}\Big|\E_{\mu}(F_k-F|\mathcal{I}(S_k))(x)\E_{n\in I_{\lfloor M\rfloor}}f_{1}(T_{k}^{n^{\gamma}}x)f_{2}(T_{k}^{an^{\gamma}}x)\Big|d\mu(x)\notag
	\\ & +
	\int_{X}\limsup_{M\to\infty}\Big|\E_{\mu}(F|\mathcal{I}(S_{\infty}))(x)\E_{n\in I_{\lfloor M\rfloor}}f_{1}(T_{k}^{n^{\gamma}}x)f_{2}(T_{k}^{an^{\gamma}}x) \notag\\ & \hspace{3cm} -
\E_{\mu}(F|\mathcal{I}(S_{\infty}))(x)\frac{1}{M}\int_{0}^{M}f_{1}(T_{k}^{t^{\gamma}}x)f_{2}(T_{k}^{at^{\gamma}}x)dt
	\Big|d\mu(x)\notag
	\\ \overset{{\eqref{eq65};}\atop\eqref{eq71}\ \text{or}\ \eqref{eq105}}\le &
	\norm{F_{k}-F}_{2}+\norm{\E_{\mu}(F|\mathcal{I}(S_k))-\E_{\mu}(F|\mathcal{I}(S_{\infty}))}_{2}+\frac{8}{4^k}.\notag
\end{align}
Note that by Theorem \ref{thm6} and Lemma \ref{lem1}, for any $k\in \N$ and $\mu$-a.e. $x\in X$, we have $$\lim_{M\to\infty}\frac{1}{M}\int_{0}^{M}f_{1}(T^{t}x)f_{2}(T^{at}x)dt=\lim_{M\to\infty}\frac{1}{M}\int_{0}^{M}f_{1}(T_{k}^{t^{\gamma}}x)f_{2}(T_{k}^{at^{\gamma}}x)dt.$$
Therefore, by \eqref{eq73}, \eqref{eq68} and \eqref{eq69}, we have
$$L(f_1,f_2,f_3)=\E_{\mu}(F|\mathcal{I}(S_{\infty}))(x)\lim\limits_{M\to\infty}\frac{1}{M}\int_{0}^{M}f_{1}(T^{t}x)f_{2}(T^{at}x)dt$$ almost everywhere.

Now, it is left to verify $\mathcal{I}(S_{\infty})=\mathcal{I}((S^t)_{t\in\R})\ (\text{mod}\ \mu)$ and $\E_{\mu}(F|\mathcal{I}((S^t)_{t\in\R}))=\E_{\mu}(f_{3}|\mathcal{I}((S^t)_{t\in\R}))$ almost everywhere.

Let $X\to \mathcal{M}(X,\X),x\mapsto\mu_x$ be the disintegration of $\mu$ with respect to $\mathcal{I}((S^t)_{t\in\R})$. Then for $\mu$-a.e. $y\in X$, $\mu_y$ is $(S^t)_{t\in\R}$-invariant. By \eqref{eq68}, for $\mu$-a.e. $y\in X$, we have
\begin{align}
	& \E_{\mu}(F|\mathcal{I}((S^t)_{t\in\R}))(y)=\int_{X}F(x)d\mu_{y}(x)=\lim_{k\to\infty}\int_{X}\int_{0}^{1}f_{3}(S_{k}^{t}x)dtd\mu_{y}(x) \notag
	\\ = &
	\lim_{k\to\infty}\int_{0}^{1}\int_{X}f_{3}(S_{k}^{t}x)d\mu_{y}(x)dt=\int_{X}f_{3}(x)d\mu_{y}(x)=\E_{\mu}(f_{3}|\mathcal{I}((S^t)_{t\in\R}))(y). \label{eq74}
\end{align}

By the definitions of $\mathcal{I}(S_{\infty}),S_1,S_2,\ldots$, we have $\mathcal{I}(S_{\infty})\supset\mathcal{I}((S^t)_{t\in\R})(\text{mod}\ \mu)$. To verify that $\mathcal{I}(S_{\infty})\subset\mathcal{I}((S^t)_{t\in\R})(\text{mod}\ \mu)$, it suffices to show that for any $A\in \mathcal{I}(S_{\infty})$ and any $t\in \R$, $\displaystyle\mu(S^{t}A\backslash A)=\int_{X}1_{S^{t}A}(1_{S^{t}A}-1_{A})d\mu=0$.
Note that for any $t\in\R,k\in\N$ and any $A\in \mathcal{I}(S_{\infty})$, by the definition of $\mathcal{I}(S_{\infty})$, we have $$\mu(S^{t}A\backslash A)=\mu(S^{\text{sgn}(t)(|t|-\lfloor |t|/\delta_{k}^{s}\rfloor \delta_{k}^{s})}A\backslash A),$$ where $\text{sgn}(t)$ denotes the sign of $t$.
This means that if for any $A\in \mathcal{I}(S_{\infty})$ and any real sequence $\{c_n\}_{n\in\N}$ with $\lim_{n\to\infty}c_n=0$, the following holds: $$\lim_{n\to\infty}\int_{X}1_{S^{c_n}A}1_{A}d\mu=\mu(A),$$
then it can be concluded that $\mathcal{I}(S_{\infty})\subset\mathcal{I}((S^t)_{t\in\R})(\text{mod}\ \mu)$.
Note that $\mathcal{A}((S^t)_{t\in \R})$ is dense in $L^{2}(\mu)$. Then it is left to show that for any real-valued $g\in \mathcal{A}((S^t)_{t\in \R})$ and any real sequence $\{c_n\}_{n\in\N}$ with $\lim_{n\to\infty}c_n=0$, $$\lim_{n\to\infty}\int_{X}g(S^{c_n}x)g(x)d\mu(x)=\int_{X}(g(x))^{2}d\mu(x).$$

Choose $\rho>0$  and  real-valued $g\in \mathcal{A}((S^t)_{t\in \R})$ arbitrarily. By (\textbf{S2}), there exists $n_0\in \N$ such that for any $n\ge n_0$, and we have
$$\Big|\int_{X}g(x)\big(g(S^{c_n}x)-g(x)\big)d\mu(x)\Big|\le \norm{g}_{\infty}\rho.$$
Since $\rho$ is arbitrary, $$\lim_{n\to\infty}\int_{X}g(S^{c_n}x)g(x)d\mu(x)=\int_{X}(g(x))^{2}d\mu(x).$$
Therefore, we have
\begin{equation}\label{eq75}
	\mathcal{I}(S_{\infty})=\mathcal{I}((S^t)_{t\in\R})\ (\text{mod}\ \mu).
\end{equation}
The whole proof is complete. \hfill $\square$
\begin{rem}
	In fact, by \eqref{eq68}, \eqref{eq69}, \eqref{eq75} and \eqref{eq74}, we have shown that
	\begin{equation}\label{eq76}
		\lim_{k\to\infty}\E_{\mu}(F_k|\mathcal{I}(S_k))(x)=\E_{\mu}(F|\mathcal{I}(S_{\infty}))(x)=\E_{\mu}(f_{3}|\mathcal{I}((S^t)_{t\in\R}))(x)
	\end{equation}
	almost everywhere.
\end{rem}

\subsection{Proof of Theorem C}
\medskip

Given that the assumptions underlying Theorem B and Theorem C are identical, we will utilize certain notions and statements from the proof of Theorem B in our proof of Theorem C.

By Theorem \ref{thm2}, to prove Theorem C, it suffices to show that for any $g_1\in \mathcal{A}((T^t)_{t\in \R})$, any $g_2\in \mathcal{A}((S^t)_{t\in \R})$ and $\mu$-a.e. $x\in X$,
\begin{equation}\label{eq28}
\begin{split}
&\lim\limits_{\min\{M_1,\ldots,M_k\}\to\infty}\frac{1}{M_{1}\cdots M_{k}}\int_{\prod_{j=1}^{k}[0, M_j]}g_{1}(T^{|{\bf t}|}x)g_{2}(S^{{|{\bf t}|}^{2}+cP({\bf t})}x)d{\bf t}\\&\hspace{5cm}=\E_{\mu}(g_1|\mathcal{I}((T^t)_{t\in\R}))(x)\E_{\mu}(g_2|\mathcal{I}((S^t)_{t\in\R}))(x).
\end{split}
\end{equation}
	
Let $g_1\in\mathcal{A}((T^t)_{t\in \R})$ and $g_2\in\mathcal{A}((S^t)_{t\in \R})$ such that $\norm{g_1}_{\infty}\le 1$ and $\norm{g_2}_{\infty}\le 1$. For any $\theta>0$, by (\textbf{S2}), there exists $\eta(\theta)\in (0,1)$ such that for any $x_0\in X_0$, $x_1\in X_1$, and any $t\in (-\eta(\theta),\eta(\theta))$, we have
\begin{equation}\label{eq12}
|g_1(T^{t}x_0)-g_1(x_0)|<\theta,\quad |g_2(S^{t}x_1)-g_2(x_1)|<\theta.
\end{equation}	

The rest of the proof is divided into three steps.

\noindent\textbf{Step I. Reduce the ergodic averages on the left-hand side of \eqref{eq28} to discrete-time form.}

Choose $\ep>0$ arbitrarily. We write $P(t_1,\ldots,t_k)$ as $\sum_{j=1}^{k}l_jt_j$. Choose $\delta$ such that $0<\delta<\eta(\ep)/16k$, $|(l_1+\cdots+l_k)c\delta|<\eta(\ep)/4$ and $c/\delta\in\Z$. Let $Q({\bf n})=Q(n_1,\ldots,n_k)= (\sum_{j=1}^{k}n_{j})^{2}+ (c/\delta)(\sum_{j=1}^{k}l_{j}n_{j})$.

Let $T_{1}=T^{\delta}, S_{1}=S^{\delta^2}$. Now, we need the following ergodic averages for an approximation argument:
\begin{equation}\label{eq53}
	 \E_{{\bf n}\in \prod_{j=1}^{k}I_{N_j}} g_{1} (T_{1}^{|{\bf n}|}x)\int_{[0,1]^{k}} g_{2}(S_{1}^{Q({\bf n})+2 |{\bf n}|{|{\bf t}|}}x)d{\bf t}.
\end{equation}

 Note that for any $x\in X_0\cap X_1$ and any $N_1,\ldots,N_k\in\N$,
\begin{align}
	& \frac{1}{\delta^{k}N_{1}\cdots N_{k}}\int_{ \prod_{j=1}^{k}[0,N_j\delta]}g_{1} (T^{{|{\bf t}|}}x)g_{2}(S^{{|{\bf t}|}^{2}+cP({\bf t})}x)d{\bf t} \notag
	\\ = & \E_{{\bf n}\in \prod_{j=1}^{k}I_{N_j}}\frac{1}{\delta^{k}}\int_{ [0,\delta]^{k}}g_{1}(T^{|{\bf n}|\delta+|{\bf t}|}x) g_{2}(S^{Q({\bf n})\delta^{2}+2\delta |{\bf n}|{|{\bf t}|}+|{\bf t}|^{2}+c(\sum_{j=1}^{k}l_{j} t_{j})}x)d{\bf t}. \label{eq54}
\end{align}

By \eqref{eq54}, the definitions of $T_1,S_1$ and the setting of $\delta$, we have that for any $x\in X_0\cap X_1$ and any $N_1,\ldots,N_k\in\N$,
\begin{align}
& \Big|\frac{1}{\delta^{k}N_{1}\cdots N_{k}}\int_{ \prod_{j=1}^{k}[0,N_j\delta]}g_{1}(T^{{|{\bf t}|}}x) g_{2}(S^{{|{\bf t}|}^{2} +cP({\bf t})} x)d{\bf t} \notag\\ &\hspace{2.0cm}-
			\E_{{\bf n}\in \prod_{j=1}^{k}I_{N_j}} g_{1} (T_{1}^{|{\bf n}|}x)\int_{ [0,1]^{k}} g_{2}(S_{1}^{Q({\bf n})+2 |{\bf n}|{|{\bf t}|}}x)d{\bf t}\Big| \notag
	\\ \overset{\eqref{eq54}} = &
	\Big| \E_{{\bf n}\in \prod_{j=1}^{k}I_{N_j}}\frac{1}{\delta^{k}}\int_{[0,\delta]^{k}}g_{1}(T^{|{\bf n}|\delta+|{\bf t}|}x) g_{2}(S^{Q({\bf n})\delta^{2}+2\delta |{\bf n}|{|{\bf t}|}+|{\bf t}|^{2}+c(\sum_{j=1}^{k}l_{j} t_{j})}x)d{\bf t} \notag\\ &\hspace{2.0cm}-
	\E_{{\bf n}\in \prod_{j=1}^{k}I_{N_j}}\frac{1}{\delta^{k}}\int_{ [0,\delta]^{k}}g_{1}(T^{|{\bf n}|\delta}x) g_{2}(S^{Q({\bf n})\delta^{2}+2\delta |{\bf n}|{|{\bf t}|}}x)d{\bf t}
	 \Big| \label{eq55}
			\\ \overset{(\textbf{S1}), (\ref{eq12})}\le &
			2\epsilon. \notag
\end{align}

For any $x\in X_0\cap X_1$, let $\displaystyle F(x)=\int_{[0,1]^{k}}g_{2}(S^{|{\bf t}|}x)d{\bf t}$. Then for any $x\in X_0\cap X_1$ and any $N_1,\ldots,N_k\in\N$, we have
\begin{align}
	&\E_{{\bf n}\in \prod_{j=1}^{k}I_{N_j}} g_{1} (T_{1}^{|{\bf n}|}x)\int_{[0,1]^{k}} g_{2}(S_{1}^{Q({\bf n})+2 |{\bf n}|{|{\bf t}|}}x)d{\bf t} \label{eq77}
	\\ = &
	\frac{N_{1}\cdots N_{k}-1}{N_{1}\cdots N_{k}}\E_{{\bf n}\in \prod_{j=1}^{k}I_{N_j}\backslash \{{\bf 0}\}} g_{1}(T_{1}^{|{\bf n}|}x) \Big(\E_{{\bf i}\in (I_{2|\bf n|})^{k}} S_{1}^{|\bf i|} F\Big)(S_{1}^{Q({\bf n})}x)+\frac{1}{N_{1}\cdots N_{k}}g_{1}(x)g_{2}(x). \notag
\end{align}

By \eqref{eq55}, \eqref{eq77} and a calculation analogous to the
\textbf{Reduction Calculation} in the proof of Theorem B, to show that the limit on the left-hand side of \eqref{eq28} exists almost everywhere, it suffices to show that the limit
\begin{equation}\label{eq29}
\lim\limits_{\min\{N_1,\ldots,N_k\}\to\infty}\E_{{\bf n}\in \prod_{j=1}^{k}I_{N_j}\backslash \{{\bf 0}\}} g_{1}(T_{1}^{|{\bf n}|}x) \Big(\E_{{\bf i}\in (I_{2|\bf n|})^{k}} S_{1}^{|\bf i|} F\Big)(S_{1}^{Q({\bf n})}x)
\end{equation}
exists for $\mu$-a.e. $x\in X$.
		
\noindent\textbf{Step II. Verify that the limit in \eqref{eq29} exists almost everywhere.}

By Theorem \ref{thm3-2}, we have
\begin{equation}\label{eq34}
  \lim\limits_{N\to\infty}\E_{{\bf n}\in (I_N)^{k}}F(S_{1}^{|{\bf n}|}x)=\E_{\mu}(F|\mathcal{I}(S_1))(x)
\end{equation}
for $\mu$-a.e. $x\in X$. For any $r,l\in\N$, let
$$E_{l}^{r}=\left \{x\in X:\sup_{\lambda\ge l}\Big|\E_{{\bf n}\in (I_{\lambda})^{k} }F(S_{1}^{|{\bf n}|}x)-\E_{\mu}(F|\mathcal{I}(S_1))(x)\Big|<\frac{1}{r}\right\}.$$
Choose $r\in \N$ arbitrarily. By \eqref{eq34}, there is $l_r\in \N$ such that for any $l\ge l_r$, $\mu(E_{l}^{r})>1-\frac{1}{r}$. Then for $\mu$-a.e. $x\in X_0\cap X_1$ and any $N_1,\cdots,N_k\ge 100l_r$, we have
\begin{align}
& \Big|\E_{{\bf n}\in \prod_{j=1}^{k}I_{N_j}\backslash \{{\bf 0}\}} g_{1}(T_{1}^{|{\bf n}|}x) \Big(\E_{{\bf i}\in (I_{2|\bf n|})^{k}} S_{1}^{|\bf i|} F-\E_{\mu}(F|\mathcal{I}(S_1)) \Big)(S_{1}^{Q({\bf n})}x)\Big| \notag
\\ \le & \frac{2}{N_{1}\cdots N_{k}}\sum_{{\bf n}\in \prod_{j=1}^{k}I_{N_j}\backslash \{{\bf 0}\},\atop |{\bf n}|\le l_r}  \Big| g_{1}(T_{1}^{|{\bf n}|}x) \Big(\E_{{\bf i}\in (I_{2|\bf n|})^{k}} S_{1}^{|\bf i|} F-\E_{\mu}(F|\mathcal{I}(S_1)) \Big)(S_{1}^{Q({\bf n})}x)\Big| \label{eq56}
\\ &\hspace{0.5cm} + \frac{2}{N_{1}\cdots N_{k}}\sum_{{\bf n}\in \prod_{j=1}^{k}I_{N_j}\backslash \{{\bf 0}\},\atop |{\bf n}|> l_r,S_{1}^{Q({\bf n})}x\in E_{l_r}^{r}} \Big| g_{1}(T_{1}^{|{\bf n}|}x) \Big(\E_{{\bf i}\in (I_{2|\bf n|})^{k}} S_{1}^{|\bf i|} F-\E_{\mu}(F|\mathcal{I}(S_1)) \Big)(S_{1}^{Q({\bf n})}x)\Big| \notag\\ &\hspace{1.0cm} +
				\frac{2}{N_{1}\cdots N_{k}}\sum_{{\bf n}\in \prod_{j=1}^{k}I_{N_j}\backslash \{{\bf 0}\},\atop |{\bf n}|> l_r, S_{1}^{Q({\bf n})}x\notin E_{l_r}^{r}} \Big| g_{1}(T_{1}^{|{\bf n}|}x) \Big(\E_{{\bf i}\in (I_{2|\bf n|})^{k}} S_{1}^{|\bf i|} F-\E_{\mu}(F|\mathcal{I}(S_1)) \Big)(S_{1}^{Q({\bf n})}x)\Big| \notag
			\\ \overset{(\textbf{S1})}\le &
			\frac{4(1+l_r)^{k}}{N_{1}\cdots N_{k}}+ \frac{2}{r}+ 4\E_{{\bf n}\in \prod_{j=1}^{k}I_{N_j}}1_{(E_{l_r}^{r})^{c}}(S_{1}^{Q({\bf n})}x). \notag
\end{align}
By \eqref{eq56} and Theorem \ref{thm3-2}, we have
\begin{align*}
& \int_X \limsup\limits_{\min\{N_1,\ldots,N_k\}\to\infty}\Big|\E_{{\bf n}\in \prod_{j=1}^{k}I_{N_j}\backslash \{{\bf 0}\}} g_{1}(T_{1}^{|{\bf n}|}x) \\ & \hspace{4cm}\cdot \Big(\E_{{\bf i}\in (I_{2|\bf n|})^{k}} S_{1}^{|\bf i|} F-\E_{\mu}(F|\mathcal{I}(S_1)) \Big)(S_{1}^{Q({\bf n})}x)\Big| d\mu(x)
\\ \le & 4\int \limsup\limits_{\min\{N_1,\ldots,N_k\}\to\infty}\E_{{\bf n}\in \prod_{j=1}^{k}I_{N_j}}1_{(E_{l_r}^{r})^{c}}(S_{1}^{Q({\bf n})}x) d\mu(x) + \frac{2}{r}
			 \le
\frac{6}{r}.
\end{align*}
Since $r$ is arbitrary, we have
\begin{equation}\label{eq78}
\lim\limits_{\min\{N_1,\ldots,N_k\}\to\infty}\Big|\E_{{\bf n}\in \prod_{j=1}^{k}I_{N_j}\backslash \{{\bf 0}\}}g_{1}(T_{1}^{|{\bf n}|}x) \Big(\E_{{\bf i}\in (I_{2|\bf n|})^{k}} S_{1}^{|\bf i|} F-\E_{\mu}(F|\mathcal{I}(S_1)) \Big)(S_{1}^{Q({\bf n})}x)\Big|=0
\end{equation}
for $\mu$-a.e. $x\in X$. By \eqref{eq78} and Theorem \ref{thm3-2}, the limit in \eqref{eq29}
exists for $\mu$-a.e. $x\in X$. 

Therefore, the limit on the left-hand side of \eqref{eq28} exists almost everywhere.

\noindent \textbf{Step III. Verify that the equality in \eqref{eq28} holds almost everywhere.}

Similar to the establishment of \eqref{eq65}, by \eqref{eq55}, \eqref{eq77}, \eqref{eq78} and Theorem \ref{thm3-2}, we have that for any $s\in\N$, there exists $\delta_s>0$ such that
\begin{equation}\label{eq79}
	\begin{split}
	& \int_{X}\limsup_{\min\{M_1,\ldots,M_k\}\to\infty}\Big|\frac{1}{M_{1}\cdots M_{k}}\int_{\prod_{j=1}^{k}[0,M_j]}g_{1}(T^{|{\bf t}|}x)g_{2}(S^{{|{\bf t}|}^{2}+cP({\bf t})}x)d{\bf t}\\ & \hspace{4.5cm}-\E_{\mu}(g_{1}|\mathcal{I}(T_s))(x)\E_{\mu}(F_{s}|\mathcal{I}(S_s))(x)\Big|d\mu(x)<\frac{1}{4^s},
	\end{split}
\end{equation}
where $T_s=T^{\delta_{s}},S_s=S^{\delta_{s}^{2}}$, and $F_{s}(x)=\displaystyle \int_{[0,1]^{k}}g_{2}(S_{s}^{|{\bf t}|}x)d{\bf t}$ for any $x\in X_0\cap X_1$. Moreover, we require that the following hold:
\begin{itemize}
	\item $\lim_{s\to\infty}\delta_{s}=0$.
	\item For any $s\in\N$, $\delta_{s}>\delta_{s+1}$ and $\delta_{s}/\delta_{s+1}\in\N$.
	\item For any $s\in\N$ and any $x\in X_{0}\cap X_1$, $|F_{s}(x)-F_{s+1}(x)|<\frac{1}{4^s}$.
\end{itemize}	

Similar to the establishments of \eqref{eq75} and \eqref{eq76}, by Martingle Theorem, we have that for $\mu$-a.e. $x\in X$, the following two equalities hold:
\begin{equation}\label{eq80}
	\lim_{s\to\infty}\E_{\mu}(g_{1}|\mathcal{I}(T_s))(x)=\E_{\mu}(g_{1}|\mathcal{I}((T^t)_{t\in\R}))(x);
\end{equation}
\begin{equation}\label{eq81}
\lim_{s\to\infty}\E_{\mu}(F_{s}|\mathcal{I}(S_s))(x)=\E_{\mu}(g_{2}|\mathcal{I}((S^t)_{t\in\R}))(x).
\end{equation}

By \eqref{eq79}, \eqref{eq80} and \eqref{eq81}, the equality in \eqref{eq28} holds almost everywhere. The whole proof is complete.
	\hfill $\square$

\section{Proof of Theorem D}\label{section-D}

In this section, we prove Theorem D. First we introduce an ergodic theorem, which can be regarded as an extension of \cite[Theorem 2]{B2}. The detailed proof of this theorem is provided in Appendix \ref{Ap1}.

\begin{thm}\label{TE}
	Given $d\in\N$, let $T_1,\ldots,T_d:X\rightarrow X$ be invertible, commuting  measure preserving transformations acting on Lebesgue probability space $(X,\X,\mu)$ and $P_1,\ldots,P_d\in \R[n]$. Then for any $f\in L^{2}(\mu)$, the limit
	$$\lim_{N\to\infty}\E_{n\in I_N}f(T_{1}^{\lfloor P_{1}(n)\rfloor}\cdots T_{d}^{\lfloor P_{d}(n)\rfloor}x)$$
	exists almost everywhere.
\end{thm}

We now proceed to prove Theorem D. The approach that we employ is analogous to that used in the proof of Theorem B.

Without loss of generality, we can suppose that $Q(0)=0$ and the leading coefficient of $Q(t)$ is $1$.

Following a similar argument as in \textbf{Step I} of the proof of Theorem B, we can find some dense subsets $\mathcal{A}((T_{1}^t)_{t\in \R})$, $\ldots$, $\mathcal{A}((T_{d}^t)_{t\in \R})$, and $\mathcal{A}((S^{\bf t})_{{\bf t}\in \R^2})$ of $L^2(\mu)$ and some full measure subsets $X_1, \ldots, X_d, X_{S}$ of $X$, where $X_S$ is $(S^{\bf t})_{{\bf t}\in \R^2}$-invariant and for each $1\le i\le d$, $X_i$ is $(T_{i}^t)_{t\in \R}$-invariant, such that the following hold:

\medskip

\noindent	(\textbf{S3}) For any ${\tilde{f}}_1\in\mathcal{A}((T_{1}^t)_{t\in \R}),\ldots,{\tilde{f}}_d\in\mathcal{A}((T_{d}^t)_{t\in \R})$,  $\tilde{g}\in\mathcal{A}((S^{\bf t})_{{\bf t}\in \R^2})$, and any $x_1\in X_1,\ldots,x_d\in X_d$, $x_S\in X_S$,  we have
$$|{\tilde{f}}_1(x_1)|\le \norm{{\tilde{f}}_1}_{\infty},\ldots,|{\tilde{f}}_d(x_d)|\le \norm{{\tilde{f}}_d}_{\infty},|\tilde{g}(x_S)|\le \norm{\tilde{g}}_{\infty}.$$

\medskip

\noindent	(\textbf{S4}) For any $\epsilon>0$, any ${\tilde{f}}_1\in\mathcal{A}((T_{1}^t)_{t\in \R})$, $\ldots$,  ${\tilde{f}}_d\in\mathcal{A}((T_{d}^t)_{t\in \R})$ and any $\tilde{g}\in \mathcal{A}((S^{\bf t})_{{\bf t}\in \R^2})$, there is $\eta>0$ such that for any $x_1\in X_1,\ldots,x_d\in X_d$, $x_S\in X_S$, any $t\in (-\eta,\eta)$ and any ${\bf t}\in\R^2$
with $|{\bf t}|<\eta$, we have
\begin{equation*} |{\tilde{f}}_{1}(T_{1}^{t}x_1)- {\tilde{f}}_{1}(x_1)|<\epsilon, \ldots, |{\tilde{f}}_{d}(T_{d}^{t}x_d)-{\tilde{f}}_{d}(x_d)| <\epsilon, \ |\tilde{g}(S^{{\bf t}}x_S)-\tilde{g} (x_S)|<\epsilon.
\end{equation*}

\subsection{Proof of $(1)$ of Theorem D}

We rewrite (\ref{TB1}) as $$\lim_{M\to\infty}\frac{1}{M}\int_{0}^{M}f_{1}(T_{1}^{(t^{\beta})^{\alpha_1/\beta}}x)\cdots f_{d}(T_{d}^{(t^{\beta})^{\alpha_d/\beta}}x)g(S^{Q(t^{\beta})\vec{e}_1}S^{t^{\beta}\vec{e}_2}x)dt.$$
By Lemma \ref{lem1}, we can assume that $\beta=1$. At this point, we need to prove the following:

Let $0<\gamma_1<\cdots<\gamma_d<1$. For any $f_1,\ldots,f_d,g\in L^{\infty}(\mu)$, we have
\begin{equation}\label{eq32}
\begin{split}
  &\lim\limits_{M\to\infty}\frac{1}{M}\int_{0}^{M} f_{1}(T_{1}^{t^{\gamma_1}}x)\cdots f_{d}(T_{d}^{t^{\gamma_d}}x)g(S^{Q(t)\vec{e}_1}S^{t\vec{e}_2}x)dt
  \\ & \hspace{4cm} =
 \E_{\mu}(g|\mathcal{I}((S^{\bf t})_{{\bf t}\in\R}))(x)\prod_{j=1}^{d}\E_{\mu}(f_{j}|\mathcal{I}((T_{j}^{t})_{t\in\R}))(x)
  \end{split}
\end{equation}
almost everywhere.

By Lemma \ref{lem4} and Theorem \ref{thm2}, to verify that the equality in \eqref{eq32} holds for $\mu$-a.e. $x\in X$, it suffices to show that for any $f_1\in\mathcal{A}((T_{1}^t)_{t\in \R}),\ldots,f_d\in\mathcal{A}((T_{d}^t)_{t\in \R})$, and any $g\in\mathcal{A}((S^{\bf t})_{{\bf t}\in \R^2})$,
\begin{equation}\label{eq33}
	\begin{split}
	&\lim\limits_{M\to\infty}\frac{1}{M}\int_{0}^{M} f_{1}(T_{1}^{t^{\gamma_1}}x)\cdots f_{d}(T_{d}^{t^{\gamma_d}}x)g(S^{Q(t)\vec{e}_1}S^{t\vec{e}_2}x)dt
	\\ & \hspace{4cm} =
\E_{\mu}(g|\mathcal{I}((S^{\bf t})_{{\bf t}\in\R}))(x)\prod_{j=1}^{d}\E_{\mu}(f_{j}|\mathcal{I}((T_{j}^{t})_{t\in\R}))(x)
	\end{split}
\end{equation}
almost everywhere.

The rest of the proof is divided into three steps.

\noindent\textbf{Step I. Reduce the left-hand side of \eqref{eq33} to some discrete-time ergodic averages.}
		
Let $f_1\in\mathcal{A}((T_{1}^t)_{t\in \R}), \ldots, f_d\in\mathcal{A}((T_{d}^t)_{t\in \R})$, and  $g\in\mathcal{A}((S^{\bf t})_{{\bf t}\in \R^2})$ such that $\norm{f_i}_{\infty}\le 1,1\le i\le d$ and $\norm{g}_{\infty}\le 1$. Then by (\textbf{S4}), we have that
for any $\theta>0$, there exists $\eta(\theta)\in (0,1)$ such that for any $x_1 \in X_1, \ldots,x_d\in X_d$, $x_S\in X_S$ and any $t\in (-\eta(\theta),\eta(\theta))$, ${\bf t}\in \R^2$ with $|{\bf t}|<\eta(\theta)$, \begin{equation}\label{eq14} |f_{1}(T_{1}^{t}x_1)-f_{1}(x_1)| <\theta, \ldots, |f_{d}(T_{d}^{t}x_d)-f_{d}(x_d)|<\theta, \ |g(S^{\bf t}x_S)-g(x_S)|<\theta.
\end{equation}
			
Choose $\ep>0$ arbitrarily. Choose $\d\in (0,1)$ such that $0<\delta<\eta(\ep)/4$ and $Q([0,\delta])\subset (-\eta(\ep)/4,\eta(\ep)/4)$.

Since $Q\in \R[t]$ with $\deg Q=s\ge 2,Q(0)=0$ and the leading coefficient of  $Q(t)$ is $1$, there exist $\tilde{P}(n), \tilde{P}_1,\ldots, \tilde{P}_{s-1}\in \R[n]$ such that the following hold:
\begin{itemize}
	\item $\tilde{P}(0)=0$.
	\item The leading coefficient of $\tilde{P}_{1}(n)$ is $s$.
	\item For each $i\in \{1,\ldots,s-1\}$, $\deg \tilde{P}_{i}(n)=s-i$.
	\item Let $P(n,t)=\tilde{P}(n)+\sum_{i=1}^{s-1}\tilde{P}_{i}(n)t^{i}$. Then for any $n\in\Z,t\in\R$, we have
	\begin{equation}\label{eq57}
	Q(n\delta+t)=Q(t)+P(n,t/\delta)\delta^{s},
	\end{equation}
\end{itemize}

Let $\tilde{T_j}=T_{j}^{\delta^{\gamma_{j}}}$ for each $j\in \{1,\ldots,d\}$, $S_1=S^{\delta^s\vec{e}_1}$ and $S_2=S^{\delta\vec{e}_2}$. Next, we aim to reduce the pointwise existence of the limit on the left-hand side of \eqref{eq33} to the pointwise existence of the following limit for an arbitrary natural number $k$: $$\lim_{N\to\infty}\E_{n\in I_N}f_{1}(\tilde{T_1}^{n^{\gamma_1}}x)\cdots f_{d}(\tilde{T_d}^{n^{\gamma_d}}x) \int_{0}^{1}(\frac{1}{k}\sum_{i=0}^{k-1}S_{1}^{i}g)(S_{2}^{n+t}x)dt.$$ 

Now, we need the following ergodic averages for an approximation argument:
\begin{equation}
\E_{n\in I_N}f_{1}(\tilde{T_1}^{n^{\gamma_1}}x)\cdots f_{d}(\tilde{T_d}^{n^{\gamma_d}}x)\int_{0}^{1} g(S_{1}^{P(n,t)}S_{2}^{n}x)dt.
\end{equation}

For any $x\in X_S\cap X_1\cap \cdots \cap X_d$ and any $N\in \N$, we have
\begin{align}
	& \frac{1}{N\delta}\int_{0}^{N\delta}f_{1}(T_{1}^{t^{\gamma_1}}x)\cdots f_{d}(T_{d}^{t^{\gamma_d}}x)g(S^{Q(t)\vec{e}_1}S^{t\vec{e}_2}x)dt\notag
	\\ \overset{\eqref{eq57}}= &
	\E_{n\in I_N}\int_{0}^{1}f_{1}((T_{1}^{\delta^{\gamma_{1}}})^{(n+t)^{\gamma_1}}x)\cdots f_{d}((T_{d}^{\delta^{\gamma_{d}}})^{(n+t)^{\gamma_d}}x)g(S^{(P(n,t)\delta^{s}+Q(\delta t))\vec{e}_1}S^{(n+ t)\delta\vec{e}_2}x)dt.\label{eq58}
\end{align}


As previously stated, invoking Mean Value Theorem, there exists a natural number $N_0\in \N$ such that for any $n\ge N_0$, any $t\in [0,1]$ and any $1\le j\le d$, the inequality $(n+t)^{\gamma_j}-n^{\gamma_j}<\delta$ holds.

By  \eqref{eq58}, the definitions of $\tilde{T_1},\ldots,\tilde{T_d},S_1,S_2$ and the setting of $\delta$, for any $x\in X_S\cap X_1\cap \cdots \cap X_d$ and any $N\ge 100N_0$, we have
\begin{align}
 & \Big|\frac{1}{N\delta}\int_{0}^{N\delta}f_{1}(T_{1}^{t^{\gamma_1}}x)\cdots f_{d}(T_{d}^{t^{\gamma_d}}x)  g(S^{Q(t)\vec{e}_1}S^{t\vec{e}_2}x)dt \notag\\ &\hspace{1cm}-\E_{n\in I_N}f_{1}(\tilde{T_1}^{n^{\gamma_1}}x)\cdots f_{d}(\tilde{T_d}^{n^{\gamma_d}}x)\int_{0}^{1} g(S_{1}^{P(n,t)}S_{2}^{n}x)dt
		 	\Big|\notag
		 	\\ \overset{\eqref{eq58}}= &
		 \Big| 	\E_{n\in I_N}\int_{0}^{1}f_{1}((T_{1}^{\delta^{\gamma_{1}}})^{(n+t)^{\gamma_1}}x)\cdots f_{d}((T_{d}^{\delta^{\gamma_{d}}})^{(n+t)^{\gamma_d}}x)g(S^{(P(n,t)\delta^{s}+Q(\delta t))\vec{e}_1}S^{(n+ t)\delta\vec{e}_2}x)dt \notag\\&\hspace{1cm}-
		 \E_{n\in I_N}f_{1}((T_{1}^{\delta^{\gamma_{1}}})^{n^{\gamma_1}}x)\cdots f_{d}((T_{d}^{\delta^{\gamma_{d}}})^{n^{\gamma_d}}x)\int_{0}^{1}g(S^{P(n,t)\delta^{s}\vec{e}_1}S^{n\delta\vec{e}_2}x)dt
		 	\Big|\label{eq18}
		 	\\ \overset{(\textbf{S3}), \atop (\ref{eq14})}\le & \frac{2N_0}{N}+2d\epsilon.\notag
\end{align}

For any $x\in X_S\cap X_1\cap \cdots \cap X_d$, let $\displaystyle F(x)=\int_{0}^{1}g(S_{1}^{t}x)dt$. By employing a calculation  analogous to \eqref{eq24}, the arguments similar to those used in establishing \eqref{eq41}, and the \textbf{Reduction Process} in the proof of Theorem B (we use Theorem \ref{TE} instead of Theorem \ref{BPET} here), we have
\begin{equation}\label{eq15}
\begin{split}
& \limsup_{N\to\infty}\Big|\E_{n\in I_N}\prod_{j=1}^{d}f_{j} (\tilde{T_j}^{n^{\gamma_j}}x)\Big(\E_{\mu}(F|\mathcal{I}(S_1))(S_{1}^{\lfloor\tilde{P}(n)\rfloor}S_{2}^{n}x)-\int_{0}^{1}g(S_{1}^{P(n,t)}S_{2}^{n}x)dt\Big)\Big|\le \ep.
\end{split}
\end{equation}
for $\mu$-a.e. $x\in X$.


By the defintion of $S_1$ and $S_2$, we have that for any $x\in X_S\cap X_1\cap \cdots \cap X_d$ and any $k\in \N$,
\begin{align}
		& \limsup_{N\to\infty}\Big|\E_{n\in I_N} f_{1}(\tilde{T_1}^{n^{\gamma_1}}x)\cdots f_{d}(\tilde{T_d}^{n^{\gamma_d}}x)\int_{0}^{1}\Big(\frac{1}{k} \sum_{i=0}^{k-1}S_{1}^{i}S_{2}^{t}g- \frac{1}{k}\sum_{i=0}^{k-1}S_{1}^{i+t}g\Big)(S_{2}^{n}x)dt\Big|
		 \overset{(\ref{eq14})} \le  2\epsilon. \label{eq17}
\end{align}

By Birkhoff's ergodic theorem and Theorem \ref{thm7}, there exists $k_0\in\N$ such that
\begin{equation}\label{eq22}
\norm{\sup_{N\ge 1}\E_{n\in I_N} \Big|\int_{0}^{1}(\frac{1}{k_0}\sum_{i=0}^{k_{0}-1} S_{1}^{i}g)(S_{1}^{t}S_{2}^{n}x)dt-\E_{\mu} (F|\mathcal{I}(S_1))(S_{2}^{n}x)\Big|}_{2}< \epsilon.
\end{equation}
By (\ref{eq22}), we have
\begin{align}
&\int_X \limsup_{N\to\infty}\Big|\E_{n\in I_N}f_{1}(\tilde{T_1}^{n^{\gamma_1}}x)\cdots f_{d}(\tilde{T_d}^{n^{\gamma_d}}x)\notag\\ & \hspace{1cm} \cdot\Big( \E_{\mu}(F|\mathcal{I}(S_1))(S_{2}^{n}x)-  \int_{0}^{1}(\frac{1}{k_0}\sum_{i=0}^{k_{0}-1} S_{1}^{i}g)(S_{1}^{t}S_{2}^{n}x)dt\Big)\Big|d\mu(x)
		\overset{(\textbf{S3}),(\ref{eq22})}\le
		\epsilon. \label{eq21}
\end{align}

By (\ref{eq18}), (\ref{eq15}), (\ref{eq21}) and (\ref{eq17}), we have
\begin{align*}
& \int_X \limsup_{K\to\infty}\sup_{M_1\ge K,\atop M_2\ge K}\Big|\frac{1}{M_1}\int_{0}^{M_1}f_{1} (T_{1}^{t^{\gamma_1}}x) \cdots f_{d}(T_{d}^{t^{\gamma_d}}x) g(S^{Q(t)\vec{e}_1}S^{t\vec{e}_2}x)dt \\ & \hspace{2cm}- \frac{1}{M_2}\int_{0}^{M_2}f_{1} (T_{1}^{t^{\gamma_1}}x)\cdots f_{d}(T_{d}^{t^{\gamma_d}}x) g(S^{Q(t)\vec{e}_1}S^{t\vec{e}_2}x)dt\Big|d\mu(x)
\\\le & \int_X \limsup_{K\to\infty}\sup_{M_1\ge K,\atop M_2\ge K}\Big|\frac{1}{\lfloor M_1/\delta\rfloor\delta} \int_{0}^{\lfloor M_1/\delta\rfloor\delta}f_{1}(T_{1}^{t^{\gamma_1}}x)\cdots f_{d}(T_{d}^{t^{\gamma_d}}x) g(S^{Q(t)\vec{e}_1}S^{t\vec{e}_2}x)dt \\ & \hspace{2cm} - \frac{1}{\lfloor M_2/\delta\rfloor\delta}\int_{0}^{\lfloor M_2/\delta\rfloor\delta}f_{1}(T_{1}^{t^{\gamma_1}}x)\cdots f_{d}(T_{d}^{t^{\gamma_d}}x) g(S^{Q(t)\vec{e}_1}S^{t\vec{e}_2}x)dt\Big|d\mu(x)
\\ \overset{{(\ref{eq18}), (\ref{eq15})}, \atop (\ref{eq21}), (\ref{eq17})} \le &
		\int_X \limsup_{K\to\infty}\sup_{M_1\ge K,\atop M_2\ge K}\Big|\E_{n\in I_{\lfloor M_1/\delta\rfloor}}f_{1} (\tilde{T_1}^{n^{\gamma_1}}x)\cdots f_{d} (\tilde{T_d}^{n^{\gamma_d}}x) \int_{0}^{1}(\frac{1}{k_0}\sum_{i=0}^{k_{0}-1}S_{1}^{i}g)(S_{2}^{n+t}x)dt \\ & \hspace{0.5cm} - \E_{n\in I_{\lfloor M_2/\delta\rfloor}}f_{1}(\tilde{T_1}^{n^{\gamma_1}}x)\cdots f_{d}(\tilde{T_d}^{n^{\gamma_d}}x)\int_{0}^{1}(\frac{1}{k_0}\sum_{i=0}^{k_{0}-1}S_{1}^{i}g)(S_{2}^{n+t}x)dt\Big|d\mu(x)+24d\epsilon.
\end{align*}

Thus, to verify that the limit on the left-hand side of \eqref{eq33} exists almost everywhere, it suffices to show that for any $k\in \N$, the limit
\begin{equation}\label{eq30}
\lim_{N\to\infty}\E_{n\in I_N}f_{1}(\tilde{T_1}^{n^{\gamma_1}}x)\cdots f_{d}(\tilde{T_d}^{n^{\gamma_d}}x) \int_{0}^{1}(\frac{1}{k}\sum_{i=0}^{k-1}S_{1}^{i}g)(S_{2}^{n+t}x)dt
\end{equation}
exists for $\mu$-a.e. $x\in X$.

\noindent\textbf{Step II. Verify that the limit in \eqref{eq30} exists almost everywhere.}

By Theorem \ref{thm4}, there exists a full measure subset $X_0$ of $X_{S}\cap X_1\cap \cdots \cap X_d$ such that for any $x\in X_0$, $k\in\N$, the limit
\begin{equation}\label{eq60}
	\lim_{M\to\infty}\frac{1}{M}\int_{0}^{M}f_{1}(\tilde{T_1}^{t^{\gamma_1}}x)\cdots f_{d}(\tilde{T_d}^{t^{\gamma_d}}x)(\frac{1}{k}\sum_{i=0}^{k-1}S_{1}^{i}g)(S_{2}^{t}x)dt
\end{equation}
exists.

 Choose $\tau>0$ arbitrarily. As mentioned before, by Mean Value Theorem, there exists $N_1\in\N$ such that for any $n\ge N_1$, any $t\in [0,1]$ and any $1\le j\le d$, $(n+t)^{\gamma_j}-n^{\gamma_j}<\eta(\tau)/\delta^{\gamma_{j}}$.

 Then for any $x\in X_0$ and any $i\in \N\cup \{0\}$, we have
\begin{align}
& \limsup_{M\to\infty}\Big| \frac{1}{M}\int_{0}^{M}f_{1}(\tilde{T_1}^{t^{\gamma_1}}x)\cdots f_{d}(\tilde{T_d}^{t^{\gamma_d}}x)(S_{1}^{i}g)(S_{2}^{t}x)dt\notag \\ & \hspace{3cm} - \E_{n\in I_{\lfloor M \rfloor}}f_{1}(\tilde{T_1}^{n^{\gamma_1}}x)\cdots f_{d}(\tilde{T_d}^{n^{\gamma_d}}x)\int_{0}^{1}(S_{1}^{i}g)(S_{2}^{n+t}x)dt\Big| \notag
			\\ \le & \limsup_{M\to\infty}\Big| \E_{n\in I_{\lfloor M \rfloor}}\int_{0}^{1} f_{1}(\tilde{T_1}^{(n+t)^{\gamma_1}}x)\cdots f_{d}(\tilde{T_d}^{(n+t)^{\gamma_d}}x)(S_{1}^{i}g)(S_{2}^{n+t}x)dt\label{eq82}\\ & \hspace{3cm} - \E_{n\in I_{\lfloor M \rfloor}}f_{1}(\tilde{T_1}^{n^{\gamma_1}}x)\cdots f_{d}(\tilde{T_d}^{n^{\gamma_d}}x)\int_{0}^{1}(S_{1}^{i}g)(S_{2}^{n+t}x)dt\Big|\notag
			\overset{(\textbf{S3}), (\ref{eq14})}\le
			d\tau.\notag
\end{align}

Note that $\tau$ is arbitrary and the limit in \eqref{eq60} exists for any $x\in X_0$. This means that for any $k\in \N$, the limit in
\eqref{eq30}
exists for $\mu$-a.e. $x\in X_0$. Therefore, the limit on the left-hand side of \eqref{eq33} exists almost everywhere.

\noindent \textbf{Step III. Verify that the equality in \eqref{eq33} holds almost everywhere.}

Note that by \eqref{eq60} and \eqref{eq82}, the limit in \eqref{eq30} exists almost everywhere and the related limit function is equal to one of \eqref{eq60} almost everywhere.

Let $f_1\in\mathcal{A}((T_{1}^t)_{t\in \R}), \ldots, f_d\in\mathcal{A}((T_{d}^t)_{t\in \R})$, and  $g\in\mathcal{A}((S^{\bf t})_{{\bf t}\in \R^2})$ such that $\norm{f_i}_{\infty}\le 1,1\le i\le d$ and $\norm{g}_{\infty}\le 1$. Similar to the establishment of \eqref{eq65}, by \eqref{eq18}, \eqref{eq15}, \eqref{eq21}, \eqref{eq17}, and Theorem \ref{thm4}, we have that for any $k\in \N$, there exist $\delta_k>0$ and $s_k\in\N$ such that
\begin{equation}\label{eq83}
	\begin{split}
	&\int_{X}\limsup_{M\to\infty}\Big|\frac{1}{M}\int_{0}^{M} f_{1}(T_{1}^{t^{\gamma_1}}x)\cdots f_{d}(T_{d}^{t^{\gamma_d}}x)g(S^{Q(t)\vec{e}_1}S^{t\vec{e}_2}x)dt
	\\ & \hspace{0.5cm} -
\E_{\mu}(\E_{i\in I_{s_k}}S_{1,k}^{i}g|\mathcal{I}((S_{2,k}^{t})_{t\in\R}))(x)\prod_{j=1}^{d}\E_{\mu}(f_{j}|\mathcal{I}((T_{j,k}^{t})_{t\in\R}))(x)
	\Big|d\mu(x)<\frac{1}{8^k},
	\end{split}
\end{equation}
where $T_{1,k}=T_{1}^{\delta_{k}^{\gamma_1}},\ldots,T_{d,k}=T_{d}^{\delta_{k}^{\gamma_d}},S_{1,k}=S^{\delta_{k}^{s}\vec{e}_1},S_{2,k}=S^{\delta_{k}\vec{e}_2}$. And we require that the following hold:
\begin{itemize}
	\item $\lim_{k\to\infty}\delta_{k}=0,\lim_{k\to\infty}s_{k}=\infty$.
	\item For any $k\in\N$, $\delta_k>\delta_{k+1},\delta_{k}/\delta_{k+1}\in\N,s_{k+1}>s_{k}$.
	\item For any $k\in\N$, we have \begin{equation}\label{eq84}
		\norm{\E_{i\in I_{s_k}}S_{1,k}^{i}g-\E_{\mu}(g|\mathcal{I}(S_{1,k}))}_{2}<\frac{1}{8^k(\norm{f_1}_{\infty}\cdots \norm{f_d}_{\infty} +1)}.
	\end{equation}
\end{itemize}
By \eqref{eq83} and \eqref{eq84}, for any $k\in\N$,
\begin{equation}\label{eq85}
	\begin{split}
	&\int_{X}\limsup_{M\to\infty}\Big|\frac{1}{M}\int_{0}^{M} f_{1}(T_{1}^{t^{\gamma_1}}x)\cdots f_{d}(T_{d}^{t^{\gamma_d}}x)g(S^{Q(t)\vec{e}_1}S^{t\vec{e}_2}x)dt
	\\ & \hspace{1cm} -
\E_{\mu}(\E_{\mu}(g|\mathcal{I}(S_{1,k}))|\mathcal{I}((S_{2,k}^{t})_{t\in\R}))(x)\prod_{j=1}^{d}\E_{\mu}(f_{j}|\mathcal{I}((T_{j,k}^{t})_{t\in\R}))(x)
	\Big|d\mu(x)<\frac{1}{4^k}.
	\end{split}
\end{equation}
Note that for any $k\in \N$, $\mathcal{I}((S_{2,k}^{t})_{t\in\R})=\mathcal{I}((S^{t\vec{e}_{2}})_{t\in\R})$ and $\mathcal{I}((T_{j,k}^{t})_{t\in\R})=\mathcal{I}((T_{j}^{t})_{t\in\R})$ for each $j\in\{1,\ldots,d\}$. Then for any $k\in\N$, we can rewrite \eqref{eq85} as
\begin{equation}\label{eq86}
\begin{split}
&\int_{X}\limsup_{M\to\infty}\Big|\frac{1}{M}\int_{0}^{M} f_{1}(T_{1}^{t^{\gamma_1}}x)\cdots f_{d}(T_{d}^{t^{\gamma_d}}x)g(S^{Q(t)\vec{e}_1}S^{t\vec{e}_2}x)dt
\\ & \hspace{1cm} -
\E_{\mu}(\E_{\mu}(g|\mathcal{I}(S_{1,k}))|\mathcal{I}((S^{t\vec{e}_{2}})_{t\in\R}))(x)\prod_{j=1}^{d}\E_{\mu}(f_{j}|\mathcal{I}((T_{j}^{t})_{t\in\R}))(x)
\Big|d\mu(x)<\frac{1}{4^k}.
\end{split}
\end{equation}

Similar to the proof of \eqref{eq75}, by Martingle Theorem, we have that for $\mu$-a.e. $x\in X$,
\begin{equation}\label{eq87}
	\lim_{k\to\infty}\E_{\mu}(g|\mathcal{I}(S_{1,k}))(x)=\E_{\mu}(g|\mathcal{I}((S^{t\vec{e}_{1}})_{t\in\R}))(x).
\end{equation}

Therefore, by \eqref{eq86} and \eqref{eq87}, the equality in \eqref{eq33} holds almost everywhere.
The whole proof is complete.
\hfill $\square$

\subsection{Proof of $(2)$ of Theorem D}

Without loss of generality, we can suppose that $c\neq 0$. By Lemma \ref{lem1}, to prove $(2)$ of Theorem D, it suffices to show that for any $f,g\in L^{\infty}(\mu)$ and $\mu$-a.e. $x\in X$, we have
\begin{equation}\label{eq88}
\begin{split}
&\lim\limits_{M\to\infty}\frac{1}{M} \int_{0}^{M}f(S^{ct\vec{e}_2}x)g(S^{Q(t)\vec{e}_1}S^{t\vec{e}_2}x)dt\\ & \hspace{4cm}=\lim\limits_{M\to\infty}\frac{1}{M} \int_{0}^{M}f(S^{ct\vec{e}_2}x)\E_{\mu}(g|\mathcal{I}((S^{t\vec{e}_1})_{t\in\R}))(S^{t\vec{e}_2}x)dt.
\end{split}
\end{equation}
By Theorem \ref{thm2}, to verify that the equality in \eqref{eq88} holds almost everywhere, it suffices to prove that for any $f,g\in \mathcal{A}((S^{\bf t})_{{\bf t}\in \R^2})$ and $\mu$-a.e. $x\in X$, we have
\begin{equation}\label{eq89}
\begin{split}
&\lim\limits_{M\to\infty}\frac{1}{M} \int_{0}^{M}f(S^{ct\vec{e}_2}x)g(S^{Q(t)\vec{e}_1}S^{t\vec{e}_2}x)dt\\ & \hspace{4cm}=\lim\limits_{M\to\infty}\frac{1}{M} \int_{0}^{M}f(S^{ct\vec{e}_2}x)\E_{\mu}(g|\mathcal{I}((S^{t\vec{e}_1})_{t\in\R}))(S^{t\vec{e}_2}x)dt.
\end{split}
\end{equation}

The rest of the proof is divided into two steps.

\noindent \textbf{Step I. Verify that the limit on the left-hand side of \eqref{eq89} exists almost everywhere.}

Let $f,g\in\mathcal{A}((S^{\bf t})_{{\bf t}\in \R^2})$ such that $\norm{f}_{\infty}\le 1$ and $\norm{g}_{\infty}\le 1$. By (\textbf{S4}), for any $\theta>0$, there exists $\eta(\theta)\in (0,1)$ such that for any $x_S\in X_S$ and any ${\bf t}\in \R^2$ with $|{\bf t}|<\eta(\theta)$, we have
\begin{equation}\label{eq27}
|f(S^{\bf t}x_S)-f(x_S)|<\theta,\quad |g(S^{\bf t}x_S)-g(x_S)|<\theta.
\end{equation}

Choose $\ep>0$ arbitrarily. Let $\d>0$ such that $0<\delta<\eta(\ep)/4$, $|c\delta|<\eta(\ep)$, and $Q([0,\delta])\subset (-\eta(\ep)/2,\eta(\ep)/2)$.
Here, for the polynomial $Q$, we still use the notation mentioned in \eqref{eq57}.

Let $S_1=S^{\delta^s\vec{e}_1},S_2=S^{\delta\vec{e}_2}$. Now, we need the following ergodic averages for an approximation argument:
\begin{equation}\label{eq62}
\E_{n\in I_N}f(S_{2}^{cn}x) \int_{0}^{1}g(S_{1}^{P(n,t)}S_{2}^{n}x)dt.
\end{equation}

For any $x\in X_S$ and any $N\in \N$, we have
\begin{align}
& \frac{1}{N\delta} \int_{0}^{N\delta} f(S^{ct\vec{e}_2}x)g(S^{Q(t)\vec{e}_1}S^{t\vec{e}_2}x)dt\notag
 	\\ \overset{\eqref{eq57}}= & \E_{n\in I_N}\frac{1}{\delta}\int_{0}^{\delta}f(S^{c(n\delta+t) \vec{e}_2}x)g(S^{(P(n,t/\delta) \delta^{s}+Q(t))\vec{e}_1}S^{(n\delta+t)\vec{e}_2}x)dt.\label{eq63}
\end{align}

By \eqref{eq63}, the definitions of $S_1,S_2$ and the setting of $\delta$, for any $x\in X_S$ and any $N\in\N$, we have
\begin{align}
		& \Big|\frac{1}{N\delta}\int_{0}^{N\delta} f(S^{ct\vec{e}_2}x)g(S^{Q(t)\vec{e}_1}S^{t\vec{e}_2}x)dt
		-
		\E_{n\in I_N}f(S_{2}^{cn}x) \int_{0}^{1}g(S_{1}^{P(n,t)}S_{2}^{n}x)dt
		\Big| \notag
		\\ \overset{\eqref{eq63}} = &
		\Big| \E_{n\in I_N}\frac{1}{\delta}\int_{0}^{\delta}f(S^{c(n\delta+t) \vec{e}_2}x)g(S^{(P(n,t/\delta) \delta^{s}+Q(t))\vec{e}_1}S^{(n\delta+t)\vec{e}_2}x)dt \label{eq64}
		\\&\hspace{1.0cm} -
		\E_{n\in I_N}\frac{1}{\delta} \int_{0}^{\delta}f(S^{cn\delta\vec{e}_2}x)g(S^{P(n,t/\delta)\delta^{s}\vec{e}_1}S^{n\delta\vec{e}_2}x)dt 	\Big|
		\overset{(\textbf{S3}), (\ref{eq27})}\le
		2\epsilon.\notag
\end{align}

For any $x\in X_S$, let $\displaystyle F(x)=\int_{0}^{1}g(S_{1}^{t}x)dt$. Similar to the establishment of \eqref{eq15}, we have that for $\mu$-a.e. $x\in X$,
\begin{align}
\limsup\limits_{N\to\infty}\Big|\E_{n\in I_N}f(S_{2}^{cn}x)\int_{0}^{1} g(S_{1}^{P(n,t)}S_{2}^{n}x)dt-\E_{n\in I_N}f(S_{2}^{cn}x)\E_{\mu}(F|\mathcal{I}(S_1))(S_{1}^{\lfloor \tilde{P}(n)\rfloor}S_{2}^{n}x)\Big|\le \ep. \label{eq90}
\end{align}

Note that for $\mu$-a.e. $x\in X_S$ and any $N\in \N$,
\begin{equation}\label{eqk}
\left|\E_{n\in I_N}f(S_{2}^{cn}x)\E_{\mu}(F|\mathcal{I}(S_1))(S_{2}^{n}x)-\E_{n\in I_N}f(S_{2}^{\lfloor cn\rfloor}x)\E_{\mu}(F|\mathcal{I}(S_1))(S_{2}^{n}x) \right|\le \ep.
\end{equation}

Based on \eqref{eq64}, \eqref{eq90} and \eqref{eqk}, we perform a calculation analogous to the \textbf{Reduction Calculation} in the proof of Theorem B. This allows us to conclude that in order to establish the existence of the limit on the left-hand side of \eqref{eq89} for almost every point, it suffices to show that the limit
 \begin{equation}\label{eq31}
 	\lim\limits_{N\to\infty}\E_{n\in I_N}f(S_{2}^{\lfloor cn\rfloor}x)\E_{\mu}(F|\mathcal{I}(S_1))(S_{2}^{n}x)
 \end{equation} exists for $\mu$-a.e. $x\in X_S$.

By Theorem \ref{thm-K}, the limit in (\ref{eq31}) exists almost everywhere. Therefore, the limit on the left-hand side of \eqref{eq89} exists almost everywhere.

\noindent \textbf{Step II. Verify that the equality in \eqref{eq89} holds almost everywhere.}

Let $f,g\in\mathcal{A}((S^{\bf t})_{{\bf t}\in \R^2})$ such that $\norm{f}_{\infty}\le 1$ and $\norm{g}_{\infty}\le 1$. Similar to the establishment of \eqref{eq65}, by \eqref{eq64}, \eqref{eq90}, \eqref{eqk} and Theorem \ref{thm-K}, for any $k\in\N$, we can find $\delta_k>0$ such that
\begin{equation}\label{eq91}
	\begin{split}
	&\int_{X}\limsup_{M\to\infty}\Big|\frac{1}{M} \int_{0}^{M}f(S^{ct\vec{e}_2}x)g(S^{Q(t)\vec{e}_1}S^{t\vec{e}_2}x)dt \\ & \hspace{3cm}-
	\E_{n\in I_{\lfloor M\rfloor}}f(S_{2,k}^{cn}x)\E_{\mu}(F_{k}|\mathcal{I}(S_{1,k}))(S_{2,k}^{n}x)
	\Big|d\mu(x)<\frac{1}{8^k},
	\end{split}
\end{equation}
where $S_{1,k}=S^{\delta_{k}^{s}\vec{e}_1},S_{2,k}=S^{\delta_{k}\vec{e}_2}$, and $\displaystyle F_{k}(x)=\int_{0}^{1}g(S_{1,k}^{t}x)dt$ for any $x\in X_S$. And we require that the following hold:
\begin{itemize}
	\item $\lim_{k\to\infty}\delta_k=0$.
	\item For any $k\in\N$, $\delta_k>\delta_{k+1}$ and $\delta_k/\delta_{k+1}\in \N$.
	\item For any $k\in \N$ and any $x\in X_0\cap X_1$,
	$|F_{k}(x)-F_{k+1}(x)|<\frac{1}{4^{k}}$.
	\item For any $k\in\N$, any $x\in X_S$ and any $t\in [0,1]$, we have
	\begin{equation}\label{eq92}
		|f(S_{2,k}^{ct}x)-f(x)|<\frac{1}{8^{k}(\norm{g}_{\infty}+1)},\quad |g(S_{2,k}^{t}x)-g(x)|<\frac{1}{8^{k}(\norm{f}_{\infty}+1)}.
	\end{equation}
\end{itemize}

For any $k\in\N$, by Birkhoff's ergodic theorem, Theorem \ref{thm7} and Theorem \ref{thm2}, there exists $s_k\in\N$ such that the following two inequalities hold:
\begin{equation}\label{eq93}
	\norm{\sup_{n\ge 1}\E_{n\in I_{N}}\Big|\E_{i\in I_{s_k}}F_{k}(S_{1,k}^{i}S_{2,k}^{n}x)-\E_{\mu}(F_{k}|\mathcal{I}(S_{1,k}))(S_{2,k}^{n}x)\Big|}_{2}<\frac{1}{8^{k}(\norm{f}_{\infty}+1)};
\end{equation}
\begin{equation}\label{eq94}
	\norm{\sup_{M\in \R_{+}}\frac{1}{M}\int_{0}^{M}\Big|\E_{i\in I_{s_k}}S_{1,k}^{i}F_{k}-\E_{\mu}(F_{k}|\mathcal{I}(S_{1,k}))\Big|(S^{t\vec{e}_2}x)dt}_{2}<\frac{1}{8^{k}(\norm{f}_{\infty}+1)}.
\end{equation}

By \eqref{eq92}, for any $k\in \N$, we have
\begin{align}
	& \int_{X}\limsup_{M\to\infty}\Big|\E_{n\in I_{\lfloor M\rfloor}}f(S_{2,k}^{cn}x)\E_{i\in I_{s_k}}F_{k}(S_{1,k}^{i}S_{2,k}^{n}x) \notag
	\\ & \hspace{1.5cm} -
	\frac{1}{M}\int_{0}^{M}f(S_{2,k}^{ct}x)\E_{i\in I_{s_k}}F_{k}(S_{1,k}^{i}S_{2,k}^{t}x)dt
	\Big|d\mu(x)
	\overset{(\textbf{S3}),\eqref{eq92}}\le
	\frac{2}{8^k}. \label{eq95}
\end{align}
By \eqref{eq91}, \eqref{eq93}, \eqref{eq95}, and \eqref{eq94}, for any $k\in\N$, we have
\begin{align}
	&\int_{X}\limsup_{M\to\infty}\Big|\frac{1}{M} \int_{0}^{M}f(S^{ct\vec{e}_2}x)g(S^{Q(t)\vec{e}_1}S^{t\vec{e}_2}x)dt \notag\\ & \hspace{3cm}-
	\frac{1}{M}\int_{0}^{M}f(S^{ct\vec{e}_2}x)\E_{\mu}(F_{k}|\mathcal{I}(S_{1,k}))(S^{t\vec{e}_2}x)dt \notag
	\\ \overset{\eqref{eq91}}\le &
	\int_{X}\limsup_{M\to\infty}\Big|\E_{n\in I_{\lfloor M\rfloor}}f(S_{2,k}^{cn}x)\E_{\mu}(F_{k}|\mathcal{I}(S_{1,k}))(S_{2,k}^{n}x) \notag\\ & \hspace{3cm}-
	\frac{1}{M}\int_{0}^{M}f(S^{ct\vec{e}_2}x)\E_{\mu}(F_{k}|\mathcal{I}(S_{1,k}))(S^{t\vec{e}_2}x)dt
	\Big|d\mu(x)+\frac{1}{8^k}\notag
	\\ \overset{\eqref{eq93}}\le &
	\int_{X}\limsup_{M\to\infty}\Big|\E_{n\in I_{\lfloor M\rfloor}}f(S_{2,k}^{cn}x)\E_{i\in I_{s_k}}F_{k}(S_{1,k}^{i}S_{2,k}^{n}x) \label{eq97}\\ & \hspace{3cm}-
	\frac{1}{M}\int_{0}^{M}f(S^{ct\vec{e}_2}x)\E_{\mu}(F_{k}|\mathcal{I}(S_{1,k}))(S^{t\vec{e}_2}x)dt
	\Big|d\mu(x)+\frac{2}{8^k}\notag
		\\ \overset{\eqref{eq95}}\le &
	\int_{X}\limsup_{M\to\infty}\Big|	\frac{1}{M}\int_{0}^{M}f(S_{2,k}^{ct}x)\E_{i\in I_{s_k}}F_{k}(S_{1,k}^{i}S_{2,k}^{t}x)dt \notag\\ & \hspace{3cm}-
	\frac{1}{M}\int_{0}^{M}f(S^{ct\vec{e}_2}x)\E_{\mu}(F_{k}|\mathcal{I}(S_{1,k}))(S^{t\vec{e}_2}x)dt
	\Big|d\mu(x)+\frac{4}{8^k}\notag
		\\ \overset{\text{Theorem \ref{thm6}}}\le &
	\int_{X}\limsup_{M\to\infty}\Big|	\frac{1}{M}\int_{0}^{M}f(S^{ct\vec{e}_2}x)\E_{i\in I_{s_k}}F_{k}(S_{1,k}^{i}S^{t\vec{e}_2}x)dt \notag\\ & \hspace{3cm}-
	\frac{1}{M}\int_{0}^{M}f(S^{ct\vec{e}_2}x)\E_{\mu}(F_{k}|\mathcal{I}(S_{1,k}))(S^{t\vec{e}_2}x)dt
	\Big|d\mu(x)+\frac{4}{8^k}\notag
		\\ \overset{\eqref{eq94}}\le &
	\frac{5}{8^k}\notag	.
\end{align}

Similar to the proof of \eqref{eq76}, by Martingle Theorem, we have that for $\mu$-a.e. $x\in X$,
\begin{equation}\label{eq98}
	\lim_{k\to\infty}\E_{\mu}(F_{k}|\mathcal{I}(S_{1,k}))(x)=\E_{\mu}(g|\mathcal{I}((S^{t\vec{e}_1})_{t\in\R}))(x).
\end{equation}
By \eqref{eq97}, \eqref{eq98} and Theorem \ref{thm2}, we have
\begin{equation*}
\begin{split}
&\int_{X}\limsup_{M\to\infty}\Big|\frac{1}{M} \int_{0}^{M}f(S^{ct\vec{e}_2}x)g(S^{Q(t)\vec{e}_1}S^{t\vec{e}_2}x)dt \\ & \hspace{3cm}-
\frac{1}{M}\int_{0}^{M}f(S^{ct\vec{e}_2}x)\E_{\mu}(g|\mathcal{I}((S^{t\vec{e}_1})_{t\in\R}))(S^{t\vec{e}_2}x)dt
\Big|d\mu(x)=0
\end{split}
\end{equation*}

Therefore, the equality in \eqref{eq89} holds almost everywhere. The whole proof is complete.
\hfill $\square$

\section{A question}\label{section-ques}
%
%
Based on the proof of Theorem B, we know that a key reason why the commutativity of measurable flows can not influence the pointwise existence of (\ref{TA1}) is the degree of $Q$. From this point of view, the following question arises spontaneously.
\begin{ques}
Let $(X,\X,\mu, (T^{t})_{t\in \R})$ and $(X,\X,\mu, (S^{t})_{t\in \R})$ be two measurable flows. Let $k_1,k_2\in \N$. Let $P_1,\ldots,P_{k_1},Q_1, \ldots,Q_{k_2}\in \R[t]$ with $$\max\{\deg P_1,\ldots,\deg P_{k_1}\}<\min \{\deg Q_1, \ldots, \deg Q_{k_2}\}.$$ Is it true that for any $f_1,\ldots,f_{k_1},g_1,\ldots,g_{k_2}\in L^{\infty}(\mu)$, the limit
\begin{equation}\label{eq5A}
	\lim\limits_{M\to\infty}\frac{1}{M}\int_{0}^{M} \prod_{i=1}^{k_1}f_{i}(T^{P_{i}(t)}x) \prod_{j=1}^{k_2}g_{j}(S^{Q_{j}(t)}x)dt
\end{equation}exists in $L^{2}(\mu)$ and almost everywhere?
\end{ques}
Austin \cite{A} proved that  when $(T^{t})_{t\in \R}$ and $(S^{t})_{t\in \R}$ are commuting $\R$-actions, the limit in \eqref{eq5A} exists in $L^{2}(\mu)$. But if there is no any commuting restriction for $\R$-actions $(T^{t})_{t\in \R}$ and $(S^{t})_{t\in \R}$, we do not know whether the limit in \eqref{eq5A} exists in $L^{2}(\mu)$.

\appendix
\section{An ergodic theorem along polynomials with real coefficients}\label{Ap1}
\begin{thm}
	Given $d\in\N$, let $T_1,\ldots,T_d:X\rightarrow X$ be invertible, commuting  measure preserving transformations acting on Lebesgue probability space $(X,\X,\mu)$ and $P_1,\ldots,P_d\in \R[n]$. Then there exists a constant $C$, depending on $d,P_1,\ldots,P_d$, such that  for any $f\in L^{2}(\mu)$,
	\begin{equation}\label{ETEA}
		\norm{\sup_{N\ge 1}\Big|\E_{n\in I_N}f(T_{1}^{\lfloor P_{1}(n)\rfloor}\cdots T_{d}^{\lfloor P_{d}(n)\rfloor}x)\Big|}_{2}\le C\norm{f}_{2}.
	\end{equation}
And the limit
\begin{equation}\label{ETEB}
	\lim_{N\to\infty}\E_{n\in I_N}f(T_{1}^{\lfloor P_{1}(n)\rfloor}\cdots T_{d}^{\lfloor P_{d}(n)\rfloor}x)
\end{equation}
exists almost everywhere.
\end{thm}
\begin{proof}
	Let $Y=X\times [0,1)^d$, $m$ be the Lebesgue measure on $[0,1)^d$ and $\mathcal{D}$ be the Borel $\sigma$-algebra of $[0,1)^d$.
	
	For any $t\in \R$, let $\vartheta(t)=t\ (\text{mod}\ 1)$. Define $S:\R^d\times Y\rightarrow Y$ by
	\begin{align*}
		& (t_1,\ldots,t_d,x,z_1,\ldots,z_d)\mapsto S^{(t_1,\ldots,t_d)}(x,z_1,\ldots,z_d)
		\\ & \hspace{5cm} :=(T_{1}^{\lfloor t_1+z_1\rfloor}\cdots T_{d}^{\lfloor t_d+z_d\rfloor}x,\vartheta(t_1+z_1),\ldots,\vartheta(t_d+z_d)).
	\end{align*}
	Then we get a measurable flow $(Y,\X\otimes \mathcal{D},\mu\times m,(S^{\bf t})_{{\bf t}\in \R^d})$. Let $f\in L^{2}(X,\mu)$ such that $f\ge 0$ almost everywhere. We define $\tilde{f}:Y\rightarrow \C$ by putting $(x,z_1,\ldots,z_d)\mapsto f(x)$ for $(\mu\times m)$-a.e. $(x,z_1,\ldots,z_d)\in Y$. For any $s\in (0,1)$, let $L_{0}(s)=[0,1-s)$ and $L_{1}(s)=[1-s,1)$. Then for any $N\in \N$ and $(\mu\times m)$-a.e. $(x,z_1,\ldots,z_d)\in Y\backslash (X\times \{(z_1,\ldots,z_d)\in [0,1)^d:z_{1}\cdots z_d=0 \})$, we have
	\begin{align}
		& \E_{n\in I_N}\tilde{f}(S^{(P_{1}(n),\ldots,P_{d}(n))}(x,z_1,\ldots,z_d))\notag
		\\ = &
		\E_{n\in I_N}f(T_{1}^{\lfloor P_{1}(n)+z_1\rfloor}\cdots T_{d}^{\lfloor P_{d}(n)+z_d\rfloor}x)\label{eq99}
		\\ = &
		\frac{1}{N}\sum_{(i_1,\ldots,i_d)\in\{0,1\}^d}\sum_{n\in I_N,\atop (\vartheta(P_{1}(n)),\ldots,\vartheta(P_{d}(n)))\in L_{i_1}(z_1)\times \cdots \times L_{i_d}(z_1)}f(T_{1}^{\lfloor P_{1}(n)\rfloor}\cdots T_{d}^{\lfloor P_{d}(n)\rfloor}(T_{1}^{i_1}\cdots T_{d}^{i_d}x))\notag.
	\end{align}
	By the fact that $f\ge 0$ almost everywhere and \eqref{eq99}, for any $N\in \N$ and $(\mu\times m)$-a.e. $(x,z_1,\ldots,z_d)\in Y\backslash (X\times \{(z_1,\ldots,z_d)\in [0,1)^d:z_{1}\cdots z_d=0 \})$, we have
	\begin{align}
		&  \sum_{(i_1,\ldots,i_d)\in\{0,1\}^d}\E_{n\in I_N}\tilde{f}(S^{(P_{1}(n),\ldots,P_{d}(n))}(T_{1}^{-i_1}\cdots T_{d}^{-i_d}x,z_1,\ldots,z_d)) \label{eq100}
		\\ & \hspace{7cm} \ge \E_{n\in I_N}f(T_{1}^{\lfloor P_{1}(n)\rfloor}\cdots T_{d}^{\lfloor P_{d}(n)\rfloor}x). \notag
	\end{align}
	
	Note that for any $n\in \Z$,
	\begin{equation}\label{eq120}
		S^{(P_{1}(n),\ldots,P_{d}(n))}=\prod_{i=1}^{D}S_{i}^{n^i},
	\end{equation}
	where $D=\max\{\deg P_1,\ldots,\deg P_2\}$ and $S_1,\ldots,S_D$ are invertible, commuting measure preserving transformations acting on $(Y,\B\otimes \mathcal{D},\mu\times m)$. By \eqref{eq100}, \eqref{eq120} and Theorem \ref{thm7}, we have
	\begin{align}
		& \norm{\sup_{N\ge 1}\Big|\E_{n\in I_N}f(T_{1}^{\lfloor P_{1}(n)\rfloor}\cdots T_{d}^{\lfloor P_{d}(n)\rfloor}x)\Big|}_{L^{2}(\mu)} \notag
		\\ \overset{\eqref{eq100}}\le &
		\sum_{(i_1,\ldots,i_d)\in\{0,1\}^d}\norm{\sup_{n\ge 1}\Big|\E_{n\in I_N}\tilde{f}(S^{(P_{1}(n),\ldots,P_{d}(n))}(T_{1}^{-i_1}\cdots T_{d}^{-i_d}x,z_1,\ldots,z_d))\Big|}_{L^{2}(\mu\times m)} \notag
		\\ \overset{\eqref{eq38}}\le &
		C\sum_{(i_1,\ldots,i_d)\in\{0,1\}^d}\norm{\tilde{f}((T_{1}^{-i_1}\cdots T_{d}^{-i_d}x,z_1,\ldots,z_d))}_{L^{2}(\mu\times m)}\label{eq101}
		\\ = &
		2^{d}C\norm{f}_{L^{2}(\mu)}, \notag
	\end{align} where $C$ is an absolute constant, depending on $d,P_1,\ldots,P_d$. By the linear property of ergodic averages, \eqref{eq101} can deduce \eqref{ETEA}.

By \eqref{ETEA}, we only need to prove that the limit in \eqref{ETEB} exists alomst everywhere for all $L^{\infty}(X,\mu)$-functions. Let $g\in L^{\infty}(X,\mu)$ with $\norm{g}_{L^{\infty}(X,\mu)}\le 1$ and we define $\tilde{g}:Y\rightarrow \C$ as we defined $\tilde{f}$ previously. Clearly, \eqref{eq99} holds for $g$ and $\tilde{g}$. By \eqref{eq120} and Theorem \ref{thm3-3}, the limit
$$\lim_{N\to\infty}\E_{n\in I_N}\tilde{g}(S^{(P_{1}(n),\ldots,P_{d}(n))}(x,z_1,\ldots,z_d))$$
exists $(\mu\times m)$-a.e. $(x,z_1,\ldots,z_d)\in Y$. So, there exist the sequences $\{z_{i,k}\}_{k\ge 1}\subset (0,1),1\le i\le d$ such that the following hold:
\begin{itemize}
	\item For each $1\le i\le d$, $z_{i,k}\to 0$ as $k\to\infty$.
	\item For $\mu$-a.e. $x\in X$ and each $k\ge 1$, the limit $$\lim_{N\to\infty}\E_{n\in I_N}\tilde{g}(S^{(P_{1}(n),\ldots,P_{d}(n))}(x,z_{1,k},\ldots,z_{d,k}))$$
	exists.
\end{itemize}
Then we have
\begin{align*}
	& \int_{X}\limsup_{N\to\infty}\sup_{N_1\ge N, \atop N_2\ge N}\Big|\E_{n\in I_{N_1}}g(T_{1}^{\lfloor P_{1}(n)\rfloor}\cdots T_{d}^{\lfloor P_{d}(n)\rfloor}x)-\E_{n\in I_{N_2}}g(T_{1}^{\lfloor P_{1}(n)\rfloor}\cdots T_{d}^{\lfloor P_{d}(n)\rfloor}x)\Big|d\mu(x)
	\\ \le &
	 2\limsup_{k\to\infty}\int_{X}\limsup_{N\to\infty}\Big|\E_{n\in I_{N}}g(T_{1}^{\lfloor P_{1}(n)\rfloor}\cdots T_{d}^{\lfloor P_{d}(n)\rfloor}x)
	 \\ & \hspace{5cm} -
	 \E_{n\in I_{N}}\tilde{g}(S^{(P_{1}(n),\ldots,P_{d}(n))}(x,z_{1,k},\ldots,z_{d,k}))\Big|d\mu(x)
	 \\ \le &
	 4(2^d-1)\limsup_{k\to\infty}\sum_{i=1}^{d}z_{i,k}\hspace{0.5cm}(\text{\eqref{eq99} and Weyl's equidistribution theorem})
	 \\ = &
	 0.
\end{align*}
By the above calculation, the limit $$\lim_{N\to\infty}\E_{n\in I_N}g(T_{1}^{\lfloor P_{1}(n)\rfloor}\cdots T_{d}^{\lfloor P_{d}(n)\rfloor}x)$$ exists almost everywhere.
The proof is complete.
\end{proof}

	 \bibliographystyle{plain}
	 \bibliography{ref}

\begin{thebibliography}{10}

\bibitem{A}
T.~Austin.
\newblock Norm convergence of continuous-time polynomial multiple ergodic
  averages.
\newblock {\em Ergodic Theory Dynam. Systems.}, 32(2):361--382, 2012.

\bibitem{BV06}
V.~Bergelson.
\newblock Combinatorial and {Diophantine} applications of ergodic theory.
  ({With} {Appendix} {A} by {A}. {Leibman}, with {Appendix} {B} by {A}. {Quas}
  and {M}. {Wierdl}).
\newblock In {\em Handbook of dynamical systems. Volume 1B}, pages 745--869.
  Amsterdam: Elsevier, 2006.

\bibitem{BVLA}
V.~Bergelson and A.~Leibman.
\newblock A nilpotent {Roth} theorem.
\newblock {\em Invent. Math.}, 147(2):429--470, 2002.

\bibitem{BLM}
V.~Bergelson, A.~Leibman, and C.~G. Moreira.
\newblock From discrete- to continuous-time ergodic theorems.
\newblock {\em Ergodic Theory Dynam. Systems.}, 32(2):383--426, 2012.

\bibitem{B88}
J.~Bourgain.
\newblock On the maximal ergodic theorem for certain subsets of the integers.
\newblock {\em Isr. J. Math.}, 61(1):39--72, 1988.

\bibitem{B2}
J.~Bourgain.
\newblock Pointwise ergodic theorems for arithmetic sets. {With} an appendix on
  return-time sequences, jointly with {H}. {Furstenberg}, {Y}. {Katznelson} and
  {D}. {S}. {Ornstein}.
\newblock {\em Publ. Math., Inst. Hautes {\'E}tud. Sci.}, 69:5--45, 1989.

\bibitem{B4}
J.~Bourgain.
\newblock Double recurrence and almost sure convergence.
\newblock {\em J. Reine Angew. Math.}, 404:140--161, 1990.

\bibitem{BMSW}
J.~Bourgain, M.~Mirek, E.~M. Stein, and J.~Wright.
\newblock On a multi-parameter variant of the {Bellow}-{Furstenberg} problem.
\newblock {\em Forum Math. Pi}, 11:64, 2023.
\newblock Id/No e23.

\bibitem{CDKR}
M.~Christ, P.~Durcik, V.~Kova{\v{c}}, and J.~Roos.
\newblock Pointwise convergence of certain continuous-time double ergodic
  averages.
\newblock {\em Ergodic Theory Dynam. Systems.}, 42(7):2270--2280, 2022.

\bibitem{DL96}
J.~Derrien and E.~Lesigne.
\newblock A pointwise polynomial ergodic theorem for exact endomorphisms and
  {K}-systems.
\newblock {\em Ann. Inst. Henri Poincar{\'e}, Probab. Stat.}, 32(6):765--778,
  1996.

\bibitem{DS}
S.~Donoso and W.~Sun.
\newblock Pointwise convergence of some multiple ergodic averages.
\newblock {\em Adv. Math.}, 330:946--996, 2018.

\bibitem{D51}
N.~Dunford.
\newblock An individual ergodic theorem for noncommutative transformations.
\newblock {\em Acta Sci. Math. Szeged.}, 14:1--4, 1951.

\bibitem{ELH}
E.~H. EL~Abdalaoui.
\newblock On the homogeneous ergodic bilinear averages with $1 $-bounded
  multiplicative weights.
\newblock arXiv:2012.06323, 2020.

\bibitem{FN16}
N.~Frantzikinakis.
\newblock Some open problems on multiple ergodic averages.
\newblock {\em Bull. Hell. Math. Soc.}, 60:41--90, 2016.

\bibitem{FN}
N.~Frantzikinakis.
\newblock Joint ergodicity of sequences.
\newblock {\em Adv. Math.}, 417:63, 2023.
\newblock Id/No 108918.

\bibitem{HN}
N.~Frantzikinakis and B.~Host.
\newblock Multiple recurrence and convergence without commutativity.
\newblock {\em J. Lond. Math. Soc., II. Ser.}, 107(5):1635--1659, 2023.

\bibitem{F77}
H.~Furstenberg.
\newblock Ergodic behavior of diagonal measures and a theorem of
  {Szemer{\'e}di} on arithmetic progressions.
\newblock {\em J. Anal. Math.}, 31:204--256, 1977.

\bibitem{G}
E.~Glasner.
\newblock {\em Ergodic theory via joinings}, volume 101 of {\em Math. Surv.
  Monogr.}
\newblock Providence, RI: American Mathematical Society (AMS), 2003.

\bibitem{GHSY}
Y.~Gutman, W.~Huang, S.~Shao, and X.~Ye.
\newblock Almost sure convergence of the multiple ergodic average for certain
  weakly mixing systems.
\newblock {\em Acta Math. Sin., Engl. Ser.}, 34(1):79--90, 2018.

\bibitem{HK05}
B.~Host and B.~Kra.
\newblock Nonconventional ergodic averages and nilmanifolds.
\newblock {\em Ann. Math. (2).}, 161(1):397--488, 2005.

\bibitem{HSY19}
W.~Huang, S.~Shao, and X.~Ye.
\newblock Pointwise convergence of multiple ergodic averages and strictly
  ergodic models.
\newblock {\em J. Anal. Math.}, 139(1):265--305, 2019.

\bibitem{HSY23}
W.~Huang, S.~Shao, and X.~Ye.
\newblock Polynomial {Furstenberg} joinings and its applications.
\newblock arXiv:2301.07881, 2023.

\bibitem{IAMMS}
A.~D. Ionescu, {\'A}.~Magyar, M.~Mirek, and T.~Z. Szarek.
\newblock Polynomial averages and pointwise ergodic theorems on nilpotent
  groups.
\newblock {\em Invent. Math.}, 231(3):1023--1140, 2023.

\bibitem{KLMP}
D.~Kosz, B.~Langowski, M.~Mirek, and P.~Plewa.
\newblock Polynomial ergodic theorems in the spirit of {Dunford} and {Zygmund}.
\newblock arXiv:2304.03802, 2023.

\bibitem{KMPW24}
D.~Kosz, M.~Mirek, S.~Peluse, and J.~Wright.
\newblock The multilinear circle method and a question of {Bergelson}.
\newblock arXiv: 2411.09478, 2024.

\bibitem{Kra06}
B.~Kra.
\newblock From combinatorics to ergodic theory and back again.
\newblock In {\em Proceedings of the international congress of mathematicians
  (ICM), Madrid, Spain, August 22--30, 2006. Volume III: Invited lectures},
  pages 57--76. Z{\"u}rich: European Mathematical Society (EMS), 2006.

\bibitem{krause2025}
B.~Krause.
\newblock A unified approach to two pointwise ergodic theorems: Double
  recurrence and return times.
\newblock arXiv: 2501.06877, 2025.

\bibitem{BMT}
B.~Krause, M.~Mirek, and T.~Tao.
\newblock Pointwise ergodic theorems for non-conventional bilinear polynomial
  averages.
\newblock {\em Ann. Math. (2)}, 195(3):997--1109, 2022.

\bibitem{P}
A.~Potts.
\newblock Multiple ergodic averages for flows and an application.
\newblock {\em Ill. J. Math.}, 55(2):589--621, 2011.

\bibitem{T08}
T.~Tao.
\newblock Norm convergence of multiple ergodic averages for commuting
  transformations.
\newblock {\em Ergodic Theory Dynam. Systems.}, 28(2):657--688, 2008.

\bibitem{V}
V.~S. Varadarajan.
\newblock Groups of automorphisms of {Borel} spaces.
\newblock {\em Trans. Amer. Math. Soc.}, 109:191--220, 1963.

\bibitem{W}
M.~N. Walsh.
\newblock Norm convergence of nilpotent ergodic averages.
\newblock {\em Ann. Math. (2).}, 175(3):1667--1688, 2012.

\bibitem{X}
R.~Xiao.
\newblock Multilinear {Wiener}-{Wintner} type ergodic averages and its
  application.
\newblock {\em Discrete Contin. Dyn. Syst.}, 44(2):425--446, 2024.

\bibitem{Xiao2024}
R.~Xiao.
\newblock Polynomial ergodic averages of measure-preserving systems acted by
  {{\(\mathbb{Z}^d\)}}.
\newblock {\em J. Differ. Equations}, 388:403--420, 2024.

\bibitem{Z07}
T.~Ziegler.
\newblock Universal characteristic factors and {Furstenberg} averages.
\newblock {\em J. Amer. Math. Soc.}, 20(1):53--97, 2007.

\bibitem{Z51}
A.~Zygmund.
\newblock An individual ergodic theorem for non-commutative transformations.
\newblock {\em Acta Sci. Math. Szeged.}, 14:103--110, 1951.

\end{thebibliography}

\end{document}